\documentclass[12pt,twoside]{amsart}
\usepackage[margin=3cm]{geometry}
\usepackage[colorlinks=false]{hyperref}
\usepackage[english]{babel}
\usepackage{graphicx,titling}
\usepackage{float}
\usepackage{amsmath,amsfonts,amssymb,amsthm}
\usepackage{lipsum}
\usepackage[T1]{fontenc}
\usepackage{fourier}
\usepackage{color}
\usepackage{esint}
\usepackage{caption}
\usepackage{ccicons}

\makeatletter
\def\blfootnote{\xdef\@thefnmark{}\@footnotetext}
\makeatother

\newcommand\ccnote{
    \blfootnote{The research is supported by the National Key R and D Program of China 2020YFA0713100 and NSFC No 11721101.}
}
\allowdisplaybreaks
\usepackage[export]{adjustbox}
\numberwithin{equation}{section}
\usepackage{setspace}
\renewcommand{\le}{\leqslant}

\renewcommand{\ge}{\geqslant}

\renewcommand{\mathbb}{\varmathbb}
\usepackage{fancyhdr}
\newtheorem{theorem}{Theorem}[section]
\newtheorem{lemma}[theorem]{Lemma}
\newtheorem{corollary}[theorem]{Corollary}
\newtheorem{proposition}[theorem]{Proposition}
\newtheorem{definition}[theorem]{Definition}
\newtheorem{remark}[theorem]{Remark}

\newtheorem{claim}[theorem]{Claim}

\usepackage{comment}
\usepackage{enumerate}
\usepackage{hyperref}
\usepackage[alphabetic, msc-links, backrefs]{amsrefs}

  \newcommand{\Ss}{\mathbf{S}}
 
 \newcommand{\RR}{\mathbf{R}}  
 \renewcommand{\SS}{\mathbf{S}}  

 \newcommand{\NN}{\mathbf{N}}

\usepackage{amsthm}

\allowdisplaybreaks

\input epsf
\def\begfig {
\begin{figure}
\small }
\def\endfig {
\normalsize
\end{figure}
}

\address{
\newline 
Yuchen Bi:Institute of Mathematics, Academy of Mathematics and Systems
  Science, University of Chinese Academy of Sciences, Beijing, 100190, P. R.
China
\newline
{\tt biyuchen15@mails.ucas.ac.cn}
\newline
 Jie Zhou:
School of Mathematical Sciences, Capital Normal University,
Beijing 100048, P.R. China.
\newline
{\tt Email:zhoujiemath@cnu.edu.cn}
}

\begin{document}

\thispagestyle{empty}

\ccnote

\vspace{1cm}


\begin{center}
\begin{huge}
\textit{Bi-Lipschitz rigidity for $L^2$-almost CMC surfaces}


\end{huge}
\end{center}

\vspace{1cm}


\begin{center}
\begin{minipage}[t]{.28\textwidth}
\begin{center}
{\large{\bf{Yuchen Bi,  Jie Zhou}}} \\
\vskip0.15cm
\end{center}
\end{minipage}
\end{center}

\vspace{1cm}

\vspace{1cm}
\vspace{1cm}
\noindent \textbf{Abstract.} \textit{For smooth surfaces properly immersed in the unit ball of $\RR^n$ with density close to one and small Willmore energy, the optimal a priori estimate(bi-Lipschitz and $W^{2,2}$ parametrization)is provided.  As an application, we discuss the quantitative rigidity for $L^2$-almost CMC surfaces.}
\vskip0.3cm

\vspace{0.5cm}
\section{Introduction}

   In geometric measure theory, one of the most important theorem is the Allard regularity theorem\cite{Al}, which says if an integral $n$-varifold $V=\underline{v}(M,\theta)$ in the unit ball $B^{n+k}$ of $\RR^{n+k}$ has density close to one and small mean curvature in the super critical sense($p>n$):
   \begin{align}\label{critical allard}
   \frac{\mu_V(B^{n+k}(0))}{\omega_n}\le 1+\delta \text{ and } \big(\int_{B^{n+k}}|\underline{H}|^{p}d\mu_V\big)^{\frac{1}{p}}\le \delta\ll 1,
   \end{align}
   then $sptV$ is a $C^{1,\alpha=1-\frac{n}{p}}$ graph with estimated norm in a small neighborhood of $0\in sptV$.  
   Based on Allard's work, Duggan  proved\cite{Du1} the graph is in fact $W^{2,p}$ and the density function $\theta\in W^{1,p}$ under the same condition.

In Allard's original work, he also considered the critical case and proved\cite[Lemma 8.4]{Al} the volume rigidity in the Hausdorff sense. That is, if the critical integral $\int|\underline{H}|^nd\mu_V$ of the generalized  mean curvature is small, then
    \begin{align*}
    \mu_V(B_1(0))\le (1+\delta)\omega_n\Rightarrow d_{H}(spt\mu_V\cap B_{1-\Psi}(0),T\cap B_{1-\Psi}(0))\le \Psi(\delta),
    \end{align*}
    where $\Psi=\Psi(\delta)\to 0$ as $\delta\to 0$ and $T$ is an $n$-dimensional  plane. In a recent work, the second author proved\cite{Z22} the Lipschitz approximation and Reifenberg condition for two dimensional varifolds with perpendicular generalized mean curvature in the critical case $\underline{H}\in L^2(d\mu_V)$, and hence bi-H\"older ($C^{\alpha=1-\Psi(\delta)}$) regularity follows by Reifenberg's topological disk theorem\cite{R60}. For higher dimension, Menne warmly reminded us that in the case of integral varifolds with critical Allard condition \eqref{critical allard},  the bi-H\"older regularity also hold, as a consequence of  his previous works \cite{M09}, \cite{M10}.   Notice that by Brakke's result\cite{B-1978}(see also \cite[Theorem 1.18]{T-2019}), the generalized mean curvature of an integral varifold is  perpendicular.
   
       To investigate whether a critical Allard  condition \eqref{critical allard} provides the bi-Lipschitz regularity for varifolds, it is natural to first ask whether such an a priori estimate holds for smooth submanifolds. This article gives an affirmative answer in dimension two. 

    \begin{theorem}[$\textbf{Bi-Lipschitz a priori estimate}$]\label{bi-lip}
 For any $n>2$,  there exists a function $\Psi(\delta)=\Psi(\delta| n)$ satisfying $\lim_{\delta\to 0}\Psi(\delta)=0$ such that the following holds.

   Let $F:(\Sigma,p)\to (B_1^n(0),0)$ be a proper immersion with $F(p)=0$ and induced metric $g=dF\otimes dF$, $\mu_g$ be the induced volume measure.  If \begin{align*}
  \mu_g(\Sigma)\le \pi(1+\delta)  \text{ and } \int_{\Sigma}|\vec{H}|^2d\mu_g\le \delta,
  \end{align*}
  then
  $$\int_{\Sigma^{\frac{1}{20}}(p)}|A|^2\le \Psi(\delta).$$
  Moreover, for $\delta$ small, there exists a topological disk $U(p)\supset \Sigma^{\frac{1}{40}}(p) ($the connected  component of  $F^{-1}(B_{\frac{1}{40}}(F(p)))$  containing  $p )$ such that $F:U(p)\to \RR^n$ is an embedding with $F(\partial U(p))\subset B_{\frac{1}{40}}^c(0)$ and $(U(p),g)$ is $1+\Psi(\delta)$ conformal bi-Lipschitz homeomorphic to $D_{\frac{1}{40}}$. More precisely,  there exists a conformal parametrization $\varphi: D_{\frac{1}{40}}\to U(p)$ such that $\varphi^{*}g=e^{2u}(dx^1\otimes dx^1+dx^2\otimes dx^2)$ and
  $$\|u\|_{C^0(D_{\frac{1}{40}})}+\|Du\|_{L^2(D_{\frac{1}{40}})}\le \Psi(\delta).$$
  \end{theorem}

    In fact, the question is reduced to surfaces with small total curvature, since it is proved \cite{T94,T95,Sun-Zhou} that surfaces with small second fundamental form in $L^2$ are bi-Lipschitz conformal to disks and is weakly compact in $W^{2,2}$. The key observation in the present paper is the weak compactness become strong when $\int_{\Sigma_i}|H_i|^2d\mu_{g_i}\to 0$, which allows us to use compactness argument similar as in\cite{W05} to get the above bi-Lipschitz a priori estimate in the critical case.  Since there is no such strong compactness result to use, the non-smooth case is much more complicate.  For more extending discussion, we refer to \cite{BZ-2022}.

     An interesting application for considering the optimal regularity and the control of the second fundamental form(they are equivalent) is the question of quantitative rigidity for surfaces with $L^2$-almost constant mean curvature(CMC). For CMC hypersurfaces,  Alexandrov's well-known theorem says the only boundary of a bounded $C^2$ domain in $\RR^{n+1}$ with constant mean curvature is the round sphere. There are many quantitative rigidity results considering bounded domain $\Omega\subset\RR^{n+1}$ with small $C^0$ or $L^2$ Alexandrov's deficit:
    \begin{align*}
    \delta_0(\Omega):=\frac{\|H-H_0\|_{C^{0}(\partial \Omega)}}{H_0}\le \varepsilon \ \text{  or   } \ \delta_2(\Omega)=\big(\frac{1}{\mathcal{H}^n(\partial \Omega)}\int_{\partial \Omega}|\frac{H}{H_0}-1|^2d\mathcal{H}^n\big)^{\frac{1}{2}}\le \varepsilon,
    \end{align*}
 where $H_0=\frac{n\mathcal{H}^n(\partial \Omega)}{(n+1)\mathcal{H}^{n+1}(\Omega)}$. For example, in the convex setting, the Hausdorff distance between $\partial \Omega$ and a single round sphere is controlled by $\delta_0(\Omega)$ in a quantitative way\cite{Di,Mo,Sc,An}. But in the non-convex case, it is shown \cite{Bu,CiMa,DeMaMiNe} there may occur disjoint union of touching balls as collapsing models. It is called bubbling phenomenon.
   \begin{figure}[H]
	\begin{center}
	
		\includegraphics[width=0.45\linewidth]{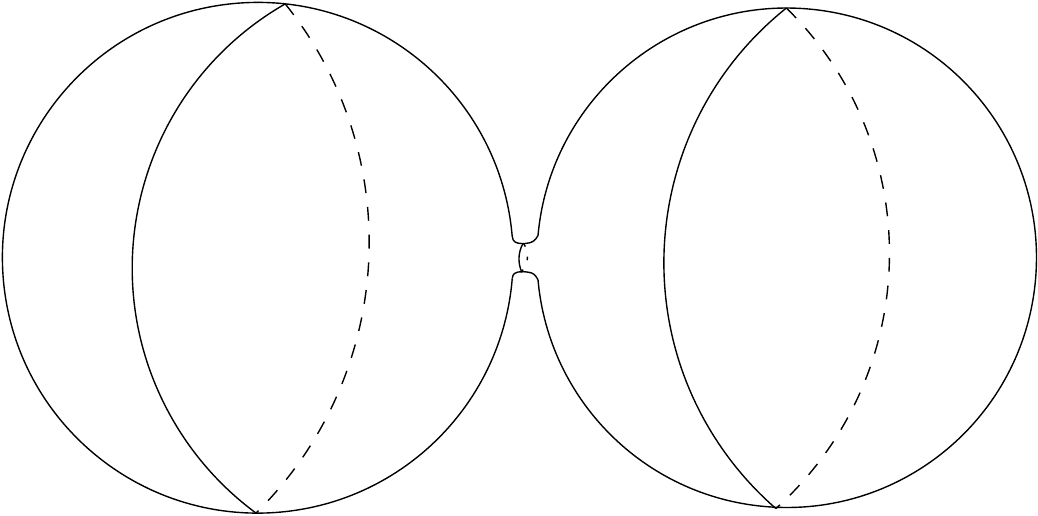}
        \caption{Constant mean curvature balls connected by small catenoid}
        \label{figure: collapse}
    \end{center}
   \end{figure}

   Recently, beginning by Ciraolo and Maggi\cite{CiMa}, there are a series of results on the quantitative rigidity of almost CMC hypersurfaces in the non-convex setting\cite{CiVe,CiVe2,CiRoVe,CiFiMaNo,DeMa}. These works dealt with hypersurfaces in the Euclidean space, Hyperbolic space, general space form and even with non-local mean curvature. The distance between $\Omega$ to the bubbling model is dominated by a precise power of $\delta_0(\Omega)$. Relaxing $\delta_0(\Omega)\le \delta$ to a weaker mean convexity assumption $H_{\partial\Omega_i}\ge \kappa n$, Delgadino, Maggi, Mihaila and Neumayer proved\cite{DeMaMiNe} a sequence of crystals $\Omega_i$ with $\delta_2(\Omega_i)\to 0$ will converge in the Hausdorff sense to a collection of touching balls with the same radius. In the present paper, as an application of our optimal a priori estimate, we considering an $L^2$-condition similar to $\delta_2(\Omega)$(see Lemma \ref{check condition} and Corollary \ref{DMMN setting} for detailed discussion)  and get the following Lipschitz quantitative rigidity for $L^2$-almost CMC surfaces.
   \begin{theorem}\label{main theorem}
   For any $\alpha \in (0,\frac{1}{2})$, there exists $\varepsilon=\varepsilon(\alpha)>0$ such that if $F:\Sigma\to {\RR}^3$ is an immersion of a closed smooth surface satisfying
\begin{align*}
\int_{\Sigma}|H-\bar{H}|^2d\mu_g\le \varepsilon \text{ and }
 \int_{\Sigma}|H|^2d\mu_g\le 32\pi(1-\alpha),
\end{align*}
where $H=\langle \vec{H},\vec{N}\rangle$ and $\bar{H}=\frac{1}{\mu_g(\Sigma)}\int_{\Sigma}Hd\mu_{g}$, then
 $$\Sigma\approx S^2\textbf{ $($Topological  Rigidity$)$}.$$

Moreover, after scaling such that $Area(\Sigma)=4\pi$, there exists $\bar{F}:{\SS}^2\to F(\Sigma)\subset {\RR}^3$ with  $d\bar{F}\otimes d\bar{F}=e^{2u}g_{{\SS}^2}$ such that after appropriate translations and rotations,
\begin{align*}
\|\bar{F}-id_{{\SS}^2}\|_{W^{2,2}}+\|u\|_{L^\infty}\le \Psi(\varepsilon),
\end{align*}
where $\Psi(\varepsilon)$ is a function such that $\lim_{\varepsilon\to0}\Psi(\varepsilon)=0$.
   \end{theorem}

   Here we do not need mean convexity assumption.  For  getting the topological rigidity and global bi-Lipschitz regularity, we need the Willmore energy restriction $\int_{\Sigma}|H|^2d\mu_g\le 32\pi(1-\alpha)$  to exclude the bubbling counterexample(Figure \ref{figure: collapse}).  In fact, the stereographic projection of a sequence of Lawson's minimal surfaces $\xi_{g,1}$\cite{L70}(after proper conformal transformations) converges to two spheres with mean curvature close to $2$ in the $L^2$ sense but their genuses can be arbitrary large. The reason is the Willmore energy restriction is not satisfied($\int_{\Sigma_i}|H|^2d\mu_{g_i}\to 32 \pi$ \cite{KLS}).  In the viewpoint of Allard's regularity theorem, the $L^2$-almost CMC assumption $\int_{\Sigma}|H-\bar{H}|^2d\mu_g\le \varepsilon$ is critical in dimension two and the bi-Lipschitz regularity is optimal.

    The topological rigidity is a key step for the bi-Lipschitz rigidity. In fact, if $\Sigma\approx S^2$, then Gauss-Bonnet formula and the Willmore energy restriction implies
   \begin{align*}
   \int_{\Sigma}|A^{\circ}|^2d\mu_g=\frac{1}{2}\int_{\Sigma}|H|^2d\mu_g-4\pi\chi(\Sigma)\le 8\pi(1-2\alpha),
   \end{align*}
   where $A^\circ$ is the trace free part of the second fundamental form.
   A surface is called umbilical if $A^\circ \equiv 0$.  Codazzi's classical theorem says all umbilical surfaces in ${\RR}^3$ are parts of a round sphere or a plane and hence must has constant mean curvature. Thus Alexandrov's theorem can be regarded as an extension of Codazzi's theorem. Before quantitative rigidity for almost CMC surfaces, there are quantitative rigidity for almost umbilical surfaces. That is the work\cite{dLMu,dLMu2}  of De Lellis and M\"uller in ${\RR}^3$ and the work  \cite{LaSc, LaNg} of Lamm, Sch\"atzle and Nguyen in arbitrary codimension.
   \begin{theorem}[$\textbf{De Lellis--M\"{u}ller(2005,2006); Lamm--Sch\"{a}tzle(2014)}$]\label{quantitative umbilical}
Assume $\Sigma\subset {\RR}^n$  is a topological sphere and $Area(\Sigma)=4\pi$. If
\begin{align*}
\int_{\Sigma}|A^\circ|^2d\mu_g<2e_n, \text{ where } e_n:=\left\{
\begin{aligned}
4\pi &\ \text{ for } n=3, \\
\frac{8\pi}{3} &\ \text{ for } n=4, \\
2\pi &\ \text{ for } n\ge 5,
\end{aligned}
\right. \text{ is defined as in  \cite{Sch13},}
\end{align*}
then there exists a conformal parametrization $f:{\SS}^2\to \Sigma$ such that $f^*g=e^{2u}g_{{\SS}^2}$ and
\begin{align*}
\|f-id_{{\SS}^2}\|_{W^{2,2}}+\|u\|_{L^\infty}\le C\|A^\circ\|_{L^2(\Sigma)}.
\end{align*}
\end{theorem}

In dimension $n=3$, there holds exactly $8\pi(1-2\alpha)<2e_n$ and hence the discussion above reduces our bi-Lipschitz rigidity to the topological rigidity, which is the main difficulty we need to deal with. The compactness theorem of surfaces in ${\RR}^n$ with small total curvature\cite{Sun-Zhou} is strong enough for topological rigidity, but for getting the non-concentration of the $L^2$-norm of second fundamental form from the assumption on the $L^2$-almost constant property on the mean curvature, the critical difficulty should also be taken into account. That is where our local bi-Lipszhitz regularity Theorem \ref{bi-lip} comes in.

 We also consider the higher codimension case. More precisely, since there are no natural conception of CMC surfaces in higher codimension, motivatied by \cite{LaSc}, we consider closed surfaces in ${\RR}^n$ such that the mean curvature vector is close to the position vector in the $L^2$ sense, i.e.,
 \begin{align*}
 \int_{\Sigma}|\vec{H}+c F|^2d\mu_g\le \varepsilon.
 \end{align*}
 The model equation $\vec{H}+cF=0$ is exactly the minimal surface equation in the sphere if $c\neq 0$. So, the following result should be called the bi-Lipszhitz quantitative rigidity for $L^2$-almost minimal surfaces in the sphere.
 \begin{theorem}[See Corollary \ref{quantitative minimal}]For any $W<8\pi$, there exists $\varepsilon=\varepsilon(W)>0$ such that if an embedding  $F: \Sigma \to \RR^n$ of a sphere type surface $\Sigma\cong S^2$ with induced metric $g=dF\otimes dF$ satisfies
\begin{align*}
\mu_g(\Sigma)=4\pi, \frac{1}{4}\int_{\Sigma}|\vec{H}|^2d\mu_g\le W<8\pi ,
\end{align*}
and there exists a constant $c\in \RR$ such that
\begin{align*}
\int_{\Sigma}|\vec{H}+cF|^2d\mu_g\le \varepsilon^2,
\end{align*}
 then there exists conformal diffeomorphism  $\tilde{F}: {\SS}^2\to F(\Sigma)\subset \RR^n$ such that   $d\tilde{F}\otimes d\tilde{F}=e^{2u}g_{{\SS}^2}$ and
\begin{align*}
|c-2|+\|u-1\|_{C^0({\SS}^2)}+\|\tilde{F}-id_{{\SS}^2}\|\le \Psi(\varepsilon).
\end{align*}
\end{theorem}

Different to Theorem \ref{main theorem}, in higher codimension,  we need to restrict the topology of $\Sigma$ to be a sphere. Otherwise, Lawson's minimal surfaces in $S^3\subset {\RR}^4$ with genus $g\neq 0$ are all counterexamples.  Comparing with Theorem \ref{quantitative umbilical}, we can also covering  $W\in [4\pi+e_n,8\pi)$ for $n\ge 4$ in the setting $\vec{H}$ close to $F$ in $L^2$.

The organization of the paper is as following. In section \ref{sec:Ctirtical allard}, we recall the compactness result for surfaces with small total curvature\cite{Sun-Zhou} and apply it to show the bi-Lipschitz a priori estimate. In section \ref{sec:quantitative rigidity}, we introduce the intermediate conception of $\gamma$ non-concentration of the area density to describe the collapsing phenomenon and use it to prove the bi-Lipschitz rigidity of $L^2$ almost CMC surfaces in ${\RR}^3$. In section \ref{sec: high codimension}, the higher codimension results are discussed, including a compactness theorem for surfaces in the sphere $\SS^{n-1}$.

We adopt the following notations in this article.
  \begin{align*}
  D_r(x)&=\{y\in \RR^2| |y-x|<r\}\\
  B_r(X)&=\text{ open geodesic ball in } \RR^n \text{ centered at } X \text{ with radius } r\\
  B_r^{\Sigma}(p)&= \text{ open geodesic ball in } \Sigma \text{ centered at } p \text{ with radius } r\\
  B_r^{F}(p)&=B_r(p)=F^{-1}(B_r(F(p)))\\
  \Sigma^r(p)&=\text{ the connected  component of } F^{-1}(B_r(F(p))) \text{ containing } p\\
  \mathcal{H}^2&=\text{ the two dimensional Hausdorff measure in }\RR^n\\
  {\Ss}^2&=\text{ the round sphere in some } {\RR}^3\subset {\RR}^n\\
  \Psi(\varepsilon| ...)&=\text{ a function with parameters such that } \lim_{\varepsilon\to 0}\Psi(\varepsilon|_...)=0 \text{ for fixed parameter}.
  \end{align*}


\section{Critical Allard regularity-the optimal a priori
estimate}\label{sec:Ctirtical allard}
In this section, we apply the compactness of surfaces in ${\RR}^n$ with small total curvature to prove the optimal a priori estimate for surfaces with critical Allard type condition.  We first recall the compactness in \cite{Sun-Zhou}. There \cite[Definition 1.1]{Sun-Zhou} the isothermal radius for a point $p$ on the Riemannian surface $(\Sigma,g)$ is defined by
\begin{align*}
i_g(p)=\sup\{r|\text{ there exists a topological disk } U(p)\supset B^{\Sigma}_r(p)\}.
\end{align*}
The main result is the isothermal radius estimate and the compactness theorem.
\begin{theorem}[\cite{Sun-Zhou}]\label{compactness}
For any $V>0$, there exists $\varepsilon_0(V)>0$  and $\alpha_0(V)>0$ such that for any properly immersed Riemannian surface $F: (\Sigma, g,p)\to (B_1(0),g_0,0)$  satisfying
\begin{align}\label{small totalcurvature}
\mu_g(\Sigma)\le V \text{ and }\int_{\Sigma}|A|^2d\mu_g\le \varepsilon_0(V),
\end{align}
there holds
\begin{align*}
i_g(x)\ge \alpha_0(V), \forall x\in \Sigma^{\frac{1}{2}}(p)
\end{align*}
Moreover, if there is a sequence of $F_i:(\Sigma_i,g_i, p_i)\to (B_1(0),g_0,0)$ satisfying $($\ref{small totalcurvature}$)$, then after passing to a subsequence, $F_i$ converges to some $W^{2,2}$ proper immersion $F:(\Sigma,g,p)\to (B_1,g_0,0)$ of a Riamannian surface $(\Sigma,g)$ with continuous metric $g$ in the  $\tau$-topology, that is
\begin{enumerate}
\item For any $q>1$, there exist smooth embedding $\Phi_k:\Sigma\to \Sigma_k$ such that
    \begin{align*}
    \Phi_k^{*}g_k\to g \text{ in } L^q_{loc}(d\mu_g);
    \end{align*}
\item the complex structure of $\Sigma_k$ converges to the complex structure of $\Sigma$;
\item $F_k\circ \Phi_k$ converges to $F$ weakly in $W^{2,2}_{loc}(\Sigma, {\RR}^n)$ and strongly in $W^{1,q}(\Sigma, {\RR}^n)$  for any $q>1$.
\end{enumerate}
\end{theorem}

In experience, for bi-Lipschitz regularity, we often need the smallness of the total curvature\cite{T94,T95,MS95, Sun-Zhou}. The following proposition shows, under critical Allard condition, the total curvature will not concentrate. This leads us to the bi-Lipschitz regularity  for surfaces with small mean curvature.

\begin{proposition}[\textbf{Non-concentration of the total curvature}]\label{non-concentration of total curvature}
  For  any fixed $n\ge 2$, $\gamma\in (0,\frac{1}{2})$ and arbitrary $\varepsilon\ll 1$, there exist $\delta(\varepsilon)=\delta(\varepsilon,\gamma, n)>0$  and $a(\varepsilon)=a(\varepsilon,\gamma, n)>0$ such that  the following holds.

  Let $F:(\Sigma,p)\to (B^n_1(0),0)$ be a proper immersion such that $F(p)=0$ with induced metric $g=dF\otimes dF$  and  $\mu_g$ be the induced volume measure. If
  \begin{align*}
\mu_g(\Sigma)\le 2(1-\gamma)\pi  \text{ and } \int_{\Sigma}|\vec{H}|^2d\mu_g\le \delta(\varepsilon),
  \end{align*}
  then  for $\sigma=\frac{\gamma}{20(1-\gamma)+\gamma}$ and any $q\in \Sigma^{\sigma}(p)$,
  \begin{align*}
  \frac{r_{\varepsilon}(q)}{\sigma-|F(q)|}\ge a(\varepsilon),
  \end{align*}
  where   $r_\varepsilon(q)=\sup\{r\le \sigma| \int_{B_r^{F}(q)}|A|^2d\mu_g\le \varepsilon\}$.
\end{proposition}
\begin{remark} Let $\theta(X)=\sharp\{F^{-1}(X)\}$ be the number of the inverse image of $X\in \RR^n$. Then $V=\underline{v}(F(\Sigma),\theta)$ is a $2$-rectifiable varifold in $B_1(0)$ with corresponding Radon measure $\mu_V=\theta\mathcal{H}^2\llcorner F(\Sigma)$ and $\mu_V(B_r(F(q)))=\mu_g(B_r(q))$ by the area formula. So, our conditions are equivalent to  $\mu_V(B_1(0))\le 2(1-\gamma)$ and $\int_{B_1(0)}|\underline{H}|^2d\mu_V\le \delta$. Thus by the monotonicity formula\cite{S,KS}, $F$ is embedding in a small neighborhood of $0$ with estimated size. Moreover, the density condition  $\mu_g(\Sigma)\le 2(1-\gamma)\pi<2\pi$ already occurred in \cite{Du}, where Allard's $C^{1,\alpha}$ regularity theorem is improved to $W^{2,p}$ regularity for varifolds with bounded generalized mean curvature  and density strictly less than $2$(not need to be close to $1$).
\end{remark}
\begin{proof}
We argue by contradiction. Otherwise, there exists $\varepsilon_0\ll 1$ and a sequence of proper immersions $F_i:\Sigma_i\to B_1(0)\subset \RR^n$ such that $0=F_i(p_i)\in F_i(\Sigma_i)$,
 \begin{align*}
  \mu_{g_i}(\Sigma_i)\le 2(1-\gamma)\pi \text{ and } \int_{\Sigma_i}|\vec{H}_i|^2d\mu_{g_i} \le \delta_i\to 0,
  \end{align*}
  but there exists $q_i\in \Sigma^{\sigma}_i(p_i)$ such that
  $$ \tau_i:=\frac{r_{\varepsilon_0}(q_i)}{\sigma-|F_i(q_i)|}=\inf_{q\in \Sigma_i^{\sigma}(p_i)} \frac{r_{\varepsilon_0}(q)}{\sigma-|F(q)|}\to 0.$$
  Especially, we have $r_i:=r_{\varepsilon_0}(q_i)\le \sigma\tau_i \to 0$.  Letting $Y_i=F_i(q_i)$ and  $\tilde{F}_i=\frac{F_i-Y_i}{r_i}$, then we have
  \begin{align*}
  \int_{B_{\tau_i^{-1}}^{\tilde{F}_i}(q_i)}|\vec{\tilde{H}}_i|^2d\mu_{\tilde{g}_i}
  =\int_{B_{\tau_i^{-1}r_i}^{F_i}(q_i)}|\vec{H}_i|^2d\mu_{g_i}\le \int_{B_{2\sigma}^{F_i}(p_i)}|\vec{H}_i|^2d\mu_{g_i}\to 0.
  \end{align*}
 By the monotonicity formula, we know that for any $\delta\in (0,1)$, fixed $R>0$ and $b=\frac{20(1-\gamma)}{20-19\gamma}$, there holds
 \begin{align*}
 \frac{\mu_{\tilde{g}_i}(B^{\tilde{F}_i}_R(q_i))}{\pi R^2}
 &=\frac{\mu_{g_i}(B_{R r_i}^{F_i}(q_i)}{\pi (Rr_i)^2}\\
 &\le (1+\delta)\frac{\mu_{g_i}(B_b^{F_i}(q_i))}{\pi b^2}+ C_\delta \int_{B_b^{F_i}(q_i)}|\vec{H}_i|^2d\mu_{g_i}\\
 &\le (1+\delta)\frac{\mu_{g_i}( B_{b+|Y_i|}^{F_i}(p_i))}{\pi(b+|Y_i|)^2}\big(1+\frac{|Y_i|}{b}\big)^2+\Psi(\frac{1}{i}|\delta)\\
 (\text{by} |Y_i|\le \sigma )&\le \frac{(b+\sigma)^2(1+\delta)}{b^2}\big[(1+\delta)\frac{\mu_{g_i}(\Sigma_i)}{\pi}+ C_\delta\int_{\Sigma_i}|\vec{H}_i|^2d\mu_{g_i}\big]+\Psi(\frac{1}{i}|\delta)\\
 &\le2(1-\gamma)(1+\frac{\sigma}{b})^2(1+\delta)^2+\Psi(\frac{1}{i}|\delta),
 \end{align*}
 where we use the notation $\Psi(\frac{1}{i}|\delta)$ to denote a function tending to zero as $i\to +\infty$ which may change from line to line.  Thus by choosing $\sigma=\frac{\gamma}{20(1-\gamma)+\gamma}$ and $\delta$ small enough, we know
 \begin{align*}
 \frac{\mu_{\tilde{g}_i}(B_R^{\tilde{F}_i}(q_i))}{\pi R^2}\le 2(1-\frac{\gamma}{4})<2
 \end{align*}
  for $i$ large enough.
  Denote $d_i(p)=\sigma-|F_i(p)|$ for $p\in \Sigma_i^{\sigma}$. Whenever $|\tilde{F}_i(p)|\le \frac{1}{2\tau_i}$, we have
  \begin{align*}
  \tilde{r}_{\varepsilon_0}(p)=\frac{r_{\varepsilon_0}(p)}{r_i}\ge \frac{d_i(p)}{d_i(q_i)}\ge 1-\frac{|F_i(p)-Y_i|}{d_i(q_i)}\ge 1-\frac{r_i}{2\tau_id_i(x_i)}=\frac{1}{2}.
  \end{align*}
  This implies for any $p\in  B^{\tilde{F}_i}_R(q_i)$,
  $$\int_{B_{\frac{1}{2}}^{\tilde{F}_i}(p)}|\tilde{A}_i|^2d\mu_{\tilde{g}_i}\le \varepsilon_0$$
  and (by the monotonicity formula again)
  $$\frac{\mu_{\tilde{g}_i}(B_{R}^{\tilde{F}_i}(p)}{\pi R^2}\le 4.$$
  Thus by choosing $\varepsilon_0<\varepsilon_0(4)$ for $\varepsilon_0(4)$ defined in (Theorem \ref{compactness}, see also \cite[Proposition 1.4]{Sun-Zhou}), we know there exists $\alpha(4)>0$ such that
  $$i_{\tilde{g}_i}(p)\ge \alpha(4), \forall p\in  \bar{\Sigma}^{\frac{R}{2}}_i(q_i),$$
  where $i_{\tilde{g}_i}(p)$ is the isothermal radius.  Letting $R\to \infty$, by Theorem \ref{compactness}(see also \cite[Theorem 1.3]{Sun-Zhou}) and diagonal argument, we know that there exists a proper Riemmanian immersion $\tilde{F}_\infty:(\Sigma_\infty,g_\infty, q_\infty)\to (\RR^{n},g_{euc},0)$ such that $F_\infty(q_\infty)=0$ and the sequence  $\tilde{F}_i:(\Sigma_i,\tilde{g}_i,q_i)\to (\RR^{n},g_{euc},0)$ converges to  $\tilde{F}_\infty:(\Sigma_\infty,g_\infty, q_\infty)\to (\RR^{n},g_{euc},0)$ in the $\tau$-topology, which means the complex structure of $(\Sigma_i,\tilde{g}_i)$ converges to the complex structure of $\Sigma_\infty$ and $\tilde{F}_i$ converges to $\tilde{F}_\infty$ weakly in $W^{2,2}_{loc}$. By the weak lower semi-continuity of the Willlmore energy(similar to the proof of \cite[Appendix B]{Sun-Zhou},note $|\vec{H}|^2$ depends convexly on $D^2F$),  we know
   $$\int_{B_{R}^{\tilde{F}_\infty}(q_\infty)}|\vec{H}_\infty|^2dg_\infty\le \lim_{i\to \infty}\int_{B_{\tau_i^{-1}}^{\tilde{F}_i}(q_i)}|\vec{\tilde{H}}_i|^2d\mu_{\tilde{g}_i}=0, \forall R>0.$$
   This means the limit $F_\infty:\Sigma_\infty\to \RR^n$ is a minimal surface with density
   $$\Theta(\Sigma_\infty,+\infty)=\lim_{R\to +\infty}\lim_{i\to \infty}\frac{\mu_{\tilde{g}_i}(B_R^{\tilde{F}_i}(q_i))}{\pi R^2}\le 2(1-\frac{\gamma}{4})<2.$$
   So, $F_\infty(\Sigma_\infty)$ is a plane in $\RR^{n}$(\cite[Theorem 2.1]{KLS}).
   Moreover, under the isothermal coordinates $\varphi_i: D_{1}\to B_{\alpha(4)}^{\Sigma_i}(p)$, we know
  $$\tilde{g}_i=e^{2u_i}(dx^2\otimes dx^1 +dx^2\otimes dx^2) \text{  and  } \Delta \tilde{F}_i=e^{2u_i}\vec{\tilde{H}}_i.$$
  By the compensated compactness estimate (see \cite[Lemma 3.1]{Sun-Zhou}), we know $|u_i|\le C$ and hence
  $$\int_{D_{\frac{1}{2}}}|e^{2u_i}\vec{\tilde{H}}_i|^2dx^1dx^2\le C\int_{B_{\tau_i^{-1}}(0)}|\vec{\tilde{H}}_i|^2d\mu_{\tilde{g}_i}\to 0.$$
  So, by elliptic estimate we know (Note $\Delta \tilde{F}_\infty=0$ since $\vec{H}_\infty=0$)
  $$\|\tilde{F}_i-\tilde{F}_\infty\|_{W^{2,2}(D_{\frac{1}{2}})}\le C\|\tilde{F}_i-\tilde{F}_\infty\|_{L^2(D_{\frac{1}{2}})}+\Psi(\frac{1}{i})\to 0. $$
  By dominated convergence theorem(note we have the uniformly estimate $|u_i|\le C$), this implies
  $$\int_{B_1^{\tilde{F}_\infty}(q_\infty)}|A_\infty|^2d\mu_{g_\infty}=\lim_{i\to \infty} \int_{B_1^{\tilde{F}_i}(q_i)}|\tilde{A}_i|^2d\mu_{\tilde{g}_i}=\varepsilon_0 (\text{ since } \tilde{r}_{\varepsilon_0}(q_i)=\frac{r_{\varepsilon_0}(q_i)}{r_i}=1),$$
  which contradicts to $\tilde{F}_\infty(\Sigma_\infty)$ is a plane.
 \end{proof}
 \begin{corollary}[$\textbf{Critical Allard regularity--the a priori estimate}$] \label{critical allard}

 For any $n\ge 2$,  there exists a function $\Psi(\delta)=\Psi(\delta| n)$ satisfying $\lim_{\delta\to 0}\Psi(\delta)=0$ such that the following holds.

   Let $F:(\Sigma,p)\to (B_1^n(0),0)$ be a proper immersion with $F(p)=0$ and induced metric $g=dF\otimes dF$, $\mu_g$ be the induced volume measure.   If
\begin{align*}
  \mu_g(\Sigma)\le \pi(1+\delta)  \text{ and } \int_{\Sigma}|\vec{H}|^2d\mu_g\le \delta,
  \end{align*}
  then
  $$\int_{\Sigma^{\frac{1}{20}}(p)}|A|^2\le \Psi(\delta).$$
  Moreover, for $\delta$ small, there exists a topological disk $U(p)\supset \Sigma^{\frac{1}{40}}(p)$ such that $F:U(p)\to \RR^n$ is an embedding with $F(\partial U(p))\subset B_{\frac{1}{40}}^c(0)$ and $(U(p),g)$ is $1+\Psi(\delta)$ conformal bi-Lipschitz homeomorphic to $D_{\frac{1}{40}}$, i.e., there exists parametrization $\varphi: D_{\frac{1}{40}}\to U(p)$ such that $\varphi^{*}g=e^{2u}(dx^1\otimes dx^1+dx^2\otimes dx^2)$ and
  $$\|u\|_{C^0(D_{\frac{1}{40}})}+\|Du\|_{L^2(D_{\frac{1}{40}})}\le \Psi(\delta).$$
  \end{corollary}
  \begin{proof}
  For the first part, we argue by contradiction. Otherwise, there will be a sequence of proper immersions
  $F_i:(\Sigma_i,p_i)\to (B_1^n(0),0)$ such that $$\mu_{g_i}(\Sigma_i)\le \pi(1+\delta_i)\to \pi \text{ and }
 \int_{\Sigma_i}|\vec{H}_i|^2d\mu_{g_i}=\delta_i\to 0,$$
  but for some $\varepsilon_1>0$,
  $$\int_{\Sigma^{\frac{1}{20}}_{i}(p_i)}|A_i|^2d\mu_{g_i}\ge \varepsilon_1.$$
   Noting for $\gamma_i=\frac{1}{2}(1-\delta_i)\to \frac{1}{2}$, we have $\sigma_i=\sigma(\gamma_i)\to \frac{1}{21}>\frac{1}{40}$. Thus by Proposition \ref{non-concentration of total curvature}, for any $\varepsilon \ll 1$, there exists $a(\varepsilon)>0$ such that  for $p\in \Sigma^{\frac{\sigma_i}{2}}(p_i)$,
   \begin{align*}
   \int_{B^{F_i}_{\frac{\sigma_ia(\varepsilon)}{2}}(p)}|A_i|^2d\mu_{g_i}\le \varepsilon.
   \end{align*}
   Again by \cite[Proposition 1.4]{Sun-Zhou}, the isothermal radius are uniformly bounded from below, i.e.,
   $$i_{g_i}(p)\ge \alpha \text{ for some constant } \alpha>0.$$
   Noting $\mu_{g_i}(\Sigma_i)\le \pi(1+\delta_i)\to \pi$, the same argument as in the proof of proposition \ref{non-concentration of total curvature} shows $F_i$ will converges strongly in $W^{2,2}_{loc}$ to a minimal proper immersion $F_\infty:\Sigma_\infty\to B_{\frac{2}{21}}(0)$ with $\mu_g(\Sigma_\infty)=\pi(\frac{2}{21})^2$, which means $F_\infty(\Sigma_\infty)$ is a plane. Moreover, the uniform bound of the metric and the $W^{2,2}_{loc}$ convergence of $F_i$ implies the contradiction:
    $$\varepsilon_1\le \lim_{i\to \infty} \int_{\Sigma_i^{\frac{1}{20}}(p_i)}|A_i|^2d\mu_{g_i}
    =\int_{\Sigma_\infty^{\frac{1}{20}}(p_\infty)}|A_\infty|^2d\mu_{g_\infty}=0.$$
    The second part of this corollary follows from the first part and \cite[Theorem 5.1]{Sun-Zhou}.
  \end{proof}
  \section{Quantitative rigidity for CMC surfaces-$3$ dimensional case}\label{sec:quantitative rigidity}

  We have seen the density condition
  \begin{align*}
  \Theta(\Sigma,r)=\frac{\mu_g(\Sigma\cap B_r(p))}{\pi r^2}\le 2\pi(1-\gamma)
  \end{align*}
   is involved in the Allard type local regularity theorem. It also provides a local description for the not occurring of the collapsing counterexample (Figure \ref{figure: collapse}).  When looking at Figure \ref{figure: collapse}, we get the intuition that the density will concentrate and  multiplicity two points occurs as the necks become smaller and smaller. Thus for the expected quantitative rigidity of CMC surfaces with  single sphere as the model,  similar to \cite[Proposition1.2(ii)]{CiMa},  we introduce the following conception of density non-concentration.
    \begin{definition}[$\gamma\ \mathbf{non-concentration}$]
    Suppose $\gamma\in (0,1)$ is a constant.  Assume $F:\Sigma\to {\RR}^{2+k}$ is an embedding of a surface in the Euclidean space and  $g=F^*g_{{\RR}^{2+k}}$ is the induced metric.  For a point $x\in F(\Sigma)$, we call
     \begin{align}\label{non-concentration}
    r^D_\gamma(x):=\sup\{a>0: \frac{vol_g(F^{-1}(B_r(x)))}{\pi r^2}\le 2(1-\gamma), \ \ \forall r\le a. \}
     \end{align}
    the $\gamma$ non-concentration radius of $x$.  We also call  $r^D_\gamma(\Sigma):=\inf_{x\in F(\Sigma)}r^D_\gamma(x)$ the $\gamma$ non-concentration radius of $\Sigma$.
     \end{definition}

 Here the superscript $D$ means "density". In the definition of $\gamma$ non-concentration radius, it is required the area ratio $\Theta(x,r)$ not extending $2(1-\gamma)$ for any $r\le a$. The next lemma shows if we are considering surfaces whose norm of mean curvature is close to a constant in $L^2$ sense, then we only need to require there exists a small $a$ such that $\Theta(x,a)\le 2(1-\gamma)$.
  \begin{lemma}\label{non-concentration radius}
  For any $W>0$ and $\gamma\in (0,\frac{1}{2})$, there exists $\varepsilon(\gamma)=\frac{8\pi\gamma}{32\pi(1-\gamma)+7}$ and $a(W,\gamma)=\frac{4\pi\gamma}{[32\pi(1-\gamma)+7]W}$ such that the following holds.

   Let $F:\Sigma\to {\RR}^{n}$ be an immersion of  a closed surface with induced metric $g=dF\otimes dF$ and area $\mu_g(\Sigma)=4\pi$.  If
  \begin{align}\label{non-concentration for mean curvature}
  \frac{1}{4}\int_{\Sigma}|\vec{H}|^2d\mu_g\le W, \ \  \int_{\Sigma}\big||\vec{H}|-\fint_{\Sigma}|\vec{H}|\big|^2d\mu_g\le \varepsilon^2\le \varepsilon^2(\gamma)
  \end{align}
  then
 \begin{align*}
  \int_{B_r(p)}|\vec{H}|^2d\mu_g&\le  2\varepsilon^2+\frac{16W^2}{\pi}r^2, \forall r\le a(W,\gamma), \forall p\in \Sigma.
 \end{align*}
  Moreover, if there exits $a\le a(W,\gamma)$ such that
  \begin{align}{\label{fix radius}}
  \frac{\mu_g(B_a(p))}{\pi a^2}\le 2(1-\gamma), \forall p\in \Sigma,
  \end{align}
  then
 $$r^D_{\frac{\gamma}{2}}(\Sigma)\ge a.$$
  \end{lemma}
  \begin{proof}
  By Leon Simon's  monotonicity formula\cite{S}, for any $\delta$ and $a$ to be determined and $r\le a$, we have
  \begin{align}\label{mono-ineq}
  \frac{\mu_g(B_r(p))}{ r^2}\le (1+\delta) \frac{\mu_g(B_a(p))}{a^2}+C_\delta\int_{B_a(p)}|\vec{H}|^2d\mu_g,
  \end{align}
  where $C_\delta=\frac{3}{16}+\frac{1}{4\delta}$. So, we need to estimate the Willmore energy in local. By (\ref{non-concentration for mean curvature}), we know that
  \begin{align}\label{upper bound}
  \fint_{\Sigma}|\vec{H}|d\mu_g\le \sqrt{\fint_{\Sigma}|\vec{H}|^2d\mu_g}\le \sqrt{\frac{W}{\pi}}
  \end{align}
  and for any  $\varepsilon\le \frac{\delta}{2}$,  $\rho\le \frac{\delta}{4W}$ and $p\in\Sigma$ ,
  \begin{align*}
  \int_{B_\rho(p)}|\vec{H}|^2d\mu_g&\le 2\int_{B_\rho(p)}|\vec{H}-\fint_{\Sigma}|\vec{H}|d\mu_g|^2d\mu_g+2(\fint_{\Sigma}|\vec{H}|d\mu_g)^2\mu_g(B_\rho(p))\\
  &\le  2\varepsilon^2+\frac{2W}{\pi}\mu_g(B_\rho(p)).
  \end{align*}
  Now, assume $d=diam(F(\Sigma))$. Then by \cite[Lemma 1.1]{S}, we know $\frac{\mu_{g}(\Sigma)}{d^2}\le W$. So, by the monotonicity formula again,
  \begin{align*}
  \frac{\mu_g(B_\rho(p))}{\rho^2}\le 2\frac{\mu_g(B_d(p))}{d^2}+6\int_{\Sigma}|\vec{H}|^2d\mu_g\le 2W+6W=8W.
  \end{align*}
  Substituting this into the above local estimate for the Willmore energy, we get
  \begin{align*}
  \int_{B_\rho(p)}|\vec{H}|^2d\mu_g&\le  2\varepsilon^2+\frac{2W}{\pi}\cdot 8W \rho^2\le \delta^2,
  \end{align*}
  where in the last in equality we use our choosing of $\varepsilon\le \varepsilon(\delta)=\frac{\delta}{2}$ and $\rho\le \rho(\delta, W)=\frac{\delta}{4W}$.
  Now,  we fix $\delta=\delta(\gamma)=\frac{16\pi\gamma}{32\pi(1-\gamma)+7}$. Then for  $\varepsilon(\gamma)=\frac{\delta(\gamma)}{2}=\frac{8\pi\gamma}{32\pi(1-\gamma)+7}$ and  $a(\gamma, W)=\rho(\delta(\gamma),W)=\frac{4\pi\gamma}{[32\pi(1-\gamma)+7]W}$, by (\ref{mono-ineq}), we know  that if (\ref{fix radius}) holds for some $a\le a(\gamma, W)$, then  for any $r\le a$,
  \begin{align*}
  \frac{\mu_g(B_r(p))}{\pi r^2}\le (1+\delta) \cdot 2(1-\gamma)+(\frac{3}{16\pi}+\frac{1}{4\pi\delta})\delta^2\le 2(1-\frac{\gamma}{2}).
  \end{align*}
  This completes the proof.
  \end{proof}

  In the following, we first show the quantitative rigidity for $L^2$-almost CMC surfaces under the assumption the densities of the surfaces not concentrate above $2$.   More precisely, for $ \gamma\in (0,\frac{1}{2}), W>0$ and $a>0$, we  denote
  \begin{align*}
  \mathcal{E}&(W,\gamma,a)=\bigg\{F:\Sigma\to {\RR}^{3}| \Sigma \text{ is a closed surface}, F \text{ is an embedding},g=dF\otimes dF,\\
   &\mu_g(\Sigma)=4\pi,\ \frac{1}{4}\int_{\Sigma}|\vec{H}|^2d\mu_g\le W,\ \frac{\mu_g(B_a(p))}{\pi a^2}\le 2(1-\gamma), \forall p\in \Sigma.\bigg\}
  \end{align*}
  \begin{proposition}[Quantitative rigidity for density non-concentrating surfaces] \label{convergence} For $ \gamma\in (0,\frac{1}{2}), W>0$,  $a(W,\gamma)$ as in Lemma \ref{non-concentration radius} and $a\in (0, a(W,\gamma))$, assume there is a sequence of $F_i:\Sigma_i\to {\RR}^{3}$ belonging to $\mathcal{E}(W,\gamma, a)$ such that
$$\int_{\Sigma_i}|H_i-\bar{H}_i|^2d\mu_{g_i}=\varepsilon_i^2\to 0,$$
where $H_i=\langle \vec{H}_i,\vec{N}_i\rangle$ is the mean curvature of $\Sigma_i$ with respect to a continuous unit normal vector field $\vec{N}_i$  and $\bar{H}_i=\fint_{\Sigma_i}H_id\mu_{g_i}$.  Then there exist conformal parametrizations  $\bar{F}_i:\SS^2\to F(\Sigma_i)\subset \RR^3$ with induced metrics  $d\bar{F}_i\otimes d\bar{F}_i=e^{2u_i}g_{\SS^{2}}$ such that after appropriate translations and rotations,
       $$\|\bar{F}_i-id_{\SS^2}\|_{W^{2,2}(\SS^2)}+\|u\|_{L^{\infty}(\SS^2)}\le \Psi(\varepsilon).$$
  \end{proposition}
  \begin{proof} Since $|\bar{H}_i|=|\fint_{\Sigma_i}H_i|d\mu_{g_i}\le \fint_{\Sigma_i}|\vec{H}_i|d\mu_{g_i}$, we have
  \begin{align*}
  \int_{\Sigma_i}||\vec{H}_i|-\fint_{\Sigma_i}|\vec{H}_i||^2d\mu_{g_i}\le \int_{\Sigma_i}|H_i-\bar{H}_i|^2d\mu_{g_i}.
  \end{align*}
     Thus by Lemma \ref{non-concentration radius}, we know for $i$ sufficiently large,
    \begin{align}\label{non concentration}
    r_{\frac{\gamma}{2}}^{D}(\Sigma_i)\ge a
    \end{align}
    and
    \begin{align*}
    \int_{B_s(x)}|H_i|^2d\mu_{g_i}\le 2\varepsilon_i^2+\frac{16 W}{\pi}s^2, \forall s\le a.
    \end{align*}
    For any $\varepsilon>0$, let $a(\varepsilon)=a(\varepsilon,\frac{\gamma}{2}, 3)$ and $\delta(\varepsilon)=\delta(\varepsilon, \frac{\gamma}{2},3)$ be the constants as in Proposition \ref{non-concentration of total curvature}. Then, for $ r_1(\varepsilon)=\min\{\frac{\sqrt{\delta(\varepsilon)}}{4W},a\}$ and sufficiently large $i$, we have
    \begin{align*}
    \int_{B_{r_1}(p)}|H_i|^2d\mu_{g_i}\le \delta(\varepsilon).
    \end{align*}
    Thus by Proposition \ref{non-concentration of total curvature} and scaling invariant, we know  for $a_1(\varepsilon)=a(\varepsilon)\sigma=\frac{a(\varepsilon)\gamma}{20(1-\gamma)+\gamma}$,
    there holds
    \begin{align*}
    \int_{B_{a_1r_1}(p)}|A_i|^2d\mu_{g_i}\le \varepsilon, \forall p\in \Sigma_i.
    \end{align*}

    So, by \cite[Proposition 1.4 and Theorem 5.1]{Sun-Zhou}, we get the uniformly lower bound of the isothermal radius $\inf i_{g_i}(\Sigma_i)\ge r>0$. More precisely, there exists $r\in (0,a_1r_1)$ such that  for any $p\in \Sigma_i$,  there exists a positive oriented  isothermal coordinate $\varphi_i: D_r(0)\to U_i\supset B^{\Sigma_i}_r(p)$ with $\varphi_i(0)=p$ and $\varphi_i^*g_i=e^{2u_i}(dx^1\otimes dx^1+dx^2\otimes dx^2)$ satisfying
    \begin{align*}
    \|u_i\|_{C^0(D_r)}+\|Du_i\|_{L^2(D_r)}\le \Psi(\varepsilon).
    \end{align*}
     Under these uniform isothermal coordinates, we have the mean curvature equation
    \begin{align*}
    \Delta \tilde{F}_i=\vec{H}_ie^{2u_i},
    \end{align*}
    where $\tilde{F}_i=F_i\circ \varphi_i-F_i(x_i)$.  By \cite[Section 4]{Sun-Zhou}, we know
    $$\|\tilde{F}_i\|_{W^{2,2}(D_\frac{3r}{4})}\le C(r)$$
     and hence  there exists a continuous metric $g_\infty =e^{2u_\infty}(dx\otimes dx+dy\otimes dy)$ on $D_\frac{3r}{4}$ and a $W^{2,2}$ Riemannian immersion $\tilde{F}_\infty:(D_\frac{3r}{4},g_\infty)\to \RR^{3}$ with $\|u_\infty\|_{L^{\infty}}\le \Psi(\varepsilon)$ and  such that after passing to a subsequence, $$\|u_i-u_\infty\|_{L^{q}(D_\frac{3r}{4})}\to 0\text{ and } \|\tilde{F}_i-\tilde{F}_\infty\|_{W^{1,q}(D_{\frac{3r}{4}})}\to 0, \text{ for any } q>1.$$
     Next, we are going to use the codimension one condition to show $\vec{H}_ie^{2u_i}=H_i\vec{N}_ie^{2u_i}$ converges strongly in $L^2$ to some limit $H_{\infty}\vec{N}_\infty e^{2u_\infty}$, where both $\vec{N}_i$ and $\vec{N}_\infty$ are unit vector fields.

       Since $\varphi_i$ is positive oriented, we know  $\vec{N}_i= e^{-2u_i}\tilde{F}_{ix}\times \tilde{F}_{iy}$ and  both sides are continuous.
       Define $\vec{N}_\infty=e^{-2u_\infty}\tilde{ F}_{\infty x}\times \tilde{F}_{\infty y}$. Then, by the $W^{1,q}$ convergence of $\tilde{F}_i$ and the $L^{\infty}$ bound of $u_i$ and dominated convergence theorem, we know that
     \begin{align*}
     \|\vec{N}_i-\vec{N}_\infty\|_{L^q(D_\frac{3r}{4})}\to 0, \text{ for some } q\gg 1
     \end{align*}
     and hence we can assume $\vec{N}_i$ converges to $\vec{N}_\infty$  and $|\vec{N}_\infty|=1$ almost everywhere. Now, by (\ref{upper bound}),we know $|\bar{H}_i|\le \sqrt{\frac{W}{\pi}}$. So, there exists a constant $H_\infty\in [-\sqrt{\frac{W}{\pi}},\sqrt{\frac{W}{\pi}}]$ such that $\bar{H}_i\to H_\infty$  after passing to a subsequence.    
     Again by the dominated convergence theorem, we know
     \begin{align*}\int_{D_\frac{3r}{4}}|H_ie^{2u_i}-H_\infty e^{2u_\infty}|^2dxdy&\le C\int_{D_\frac{3r}{4}}|H_i-\bar{H}_i|^2e^{2u_i}dxdy\\
     &\quad \quad \quad +2\int_{D_\frac{3r}{4}}|\bar{H}_ie^{2u_i}-H_\infty e^{2u_\infty}|^2dxdy\\
     &\to 0
     \end{align*}
     and hence
     $$\|\vec{H}_ie^{2u_i}-H_{\infty}\vec{N}_\infty e^{2u_\infty}\|_{L^2(D_\frac{3r}{4})}\le \|(H_ie^{2u_i}-H_\infty e^{2u_\infty})\|_{L^2(D_\frac{3r}{4})}+C\|\vec{N}_i-\vec{N}_\infty\|_{L^2(D_\frac{3r}{4})}\to 0.$$
     By the weakly convergence of $\tilde{F}_i$ in $W^{2,2}$, we get the equation
     $$\Delta \tilde{F}_\infty=\vec{H}_\infty e^{2u_\infty}=H_\infty \vec{N}_\infty e^{2u_\infty}=H_\infty \tilde{F}_{\infty x}\times\tilde{ F}_{\infty y}.$$
     Since $H_\infty$ is a constant and $\tilde{F}_\infty\in W^{2,2}$,  we know $H_\infty \tilde{F}_{\infty x}\times \tilde{F}_{\infty y}\in L^{q}$ for any $q>1$.
     So, by the interior elliptic estimate, we know $\tilde{F}_\infty\in W^{2,q}_{loc}(D_\frac{3r}{4})$ for any $q>1$. By Sobolev's embedding and the bootstrap argument, we know $\tilde{F}_\infty\in C^{\infty}(D_\frac{3r}{4})$.  Noting $d\tilde{F}_\infty\otimes d\tilde{F}_\infty=e^{2u_\infty}(dx\otimes dx+dy\otimes dy)$, we also know $u_\infty\in C^\infty(D_{\frac{3r}{4}}).$  Moreover, since
     \begin{align*}
     \Delta(\tilde{F}_i-\tilde{F}_\infty)=\vec{H}_ie^{2u_i}-H_{\infty}\vec{N}_\infty e^{2u_\infty},
     \end{align*}
     by the interior elliptic estimate again, we know
     \begin{align*}
     \|\tilde{F}_i-\tilde{F}_\infty\|_{W^{2,2}(D_{\frac{r}{2}})}\le C(\|\tilde{F}_i-\tilde{F}_\infty\|_{W^{1,q}(D_\frac{3r}{4})}+\|\vec{H}_ie^{2u_i}-H_{\infty}\vec{N}_\infty e^{2u_\infty}\|_{L^2(D_\frac{3r}{4})})\to 0.
     \end{align*}
     Now, note we have
         $$\int_{D_{\frac{r}{2}}}e^{2u_i}dxdy\ge \frac{1}{4} e^{-2\Psi(\varepsilon)}\pi r^2 \text{ and } vol_{g_i}(\Sigma_i)=4\pi.$$
      By \cite[Lemma A.1 and section 4]{Sun-Zhou}, we can glue finite covers of the above $\tilde{F}_\infty$ to get a smooth immersion $F_\infty:(\Sigma_\infty, g_\infty)\to \RR^{3}$ of a closed Riemann surface $(\Sigma_\infty, g_\infty)$  such that $F_\infty(\Sigma_\infty)$ has constant mean curvature and there exist diffeomorphisms  $\phi_i\in C^{\infty}(\Sigma_\infty, \Sigma_i)$   satisfying
      \begin{align*}
      \|F_i\circ \phi_i-F_\infty\|_{W^{2,2}(\Sigma_\infty)}\to 0,  \ C^{-1}g_\infty \le \phi_i^{*}g_i\le C g_\infty \text{ and }  \|\phi_i^{*}g_i-g_\infty\|_{L^q(\Sigma_\infty)}\to 0, q\gg 1.
      \end{align*}
      Denote $\hat{F}_i=F_i\circ \phi_i$. Then $\|\hat{F}_i-F_\infty\|_{C^{\alpha}(\Sigma_\infty)}\to 0$.  We are going to show $F_\infty$ is an embedding. Otherwise, there exist $p_1\neq p_2\in \Sigma_\infty$ such that $F_\infty(p_1)=F_\infty(p_2)$ and hence $|\hat{F}_i(p_1)-\hat{F}_i(p_2)|\to 0$. Setting $l=d_{g_\infty}(p_1,p_2)$ and $p_{i1}=\phi_i(p_1), p_{i2}=\phi_i(p_2)$, then
       $$
       d_{g_i}(p_{i1},p_{i2})=d_{\phi_i^{*}g_i}(p_1,p_2)\ge C^{-1}d_{g_\infty}(p_1,p_2)=C^{-1}l>0.
       $$
        For any fixed $r'\le \frac{1}{8}\min\{C^{-1}l,r\}$. Then, for $i$ large enough,  $p_{i2}\in B^{F_i}_{r'}(p_{i1}))$ since $\tau_i:=|F_i(p_{i1})-F_i(p_{i2})|\to 0$.  So, if we define $\Sigma_i^{r'}(p_{i1})$ and $\Sigma_i^{r'}(p_{i2})$ to be the connected components of $F_i^{-1}(B_{r'}(F_i(p_{i1})))$ passing through $p_{i1}$ and $p_{i2}$ respectively, then both of them are non-empty. Moreover, we claim $$\Sigma_i^{r'}(p_{i1})\cap\Sigma_i^{r'}(p_{i2})=\emptyset.$$
       In fact, if there exists $\xi_i\in \Sigma_i^{r'}(p_{i1})\cap\Sigma_i^{r'}(p_{i2})$, by  Again by\cite[Theorem 5.1]{Sun-Zhou}, we know
       \begin{align*}
       d_{g_i}(p_{i1},p_{i2})\le d_{g_i}(p_{i1},\xi_i)+d_{g_i}(p_{i2},\xi_i)\le (1+\frac{22}{20})(|F_i(p_{i1})-F_i(\xi_i)|+|F_i(p_{i2})-F_i(\xi_i)|)\le 6r',
       \end{align*}
       which contradicts to our choice of $r'\le \frac{1}{8}C^{-1}l\le \frac{1}{8}d_{g_i}(p_{i1},p_{i2})$. By the monotonicity formula, we know
       \begin{align*}
       \mathcal{H}^2(F_i(\Sigma_i^{r'}(p_{i2})))\ge (1-C\int_{B_{(1+\tau_i)r'}(F_i(p_{i2}))}|H_i|^2d\mu_{g_i})\pi (r'-\tau_i)^2\ge (1-C\delta)\pi (r')^2
       \end{align*}
       for $i$ large enough. For the same reason, $\mathcal{H}^2(F_i(\Sigma_i^{r'}(p_{i1})))\ge (1-C\delta)\pi (r')^2.$ Since $F_i$ is embedding, we know for $i$ sufficient large, there holds
       \begin{align*}
       \mathcal{H}^2(B_{r'}(F_i(p_{i1})))\ge \mathcal{H}^2(F_i(\Sigma_i^{r'}(p_i)))+\mathcal{H}^2(F_i(\Sigma_i^{r'}(p_{i2})))\ge (1-C\delta)\pi (r')^2.
       \end{align*}
       Thus
       \begin{align*}
       2(1-C\delta)\le \lim_{r'\to 0}\lim_{i\to \infty}\frac{\mathcal{H}^2(B_{r'}(F_i(p_{i1})))}{\pi (r')^2},
       \end{align*}
       which contradicts to our non-concentration condition (\ref{non concentration}) if we take $\delta$ small enough. This contradiction shows the limit $F_\infty:\Sigma_\infty\to \RR^3$ is an embedding with constant mean curvature $H_\infty$. Thus by Alexandrov's theorem, $F_\infty(\Sigma_\infty)$ must be the round sphere with constant mean curvature $H_\infty$. Moreover,
        since $\|\phi^*g_i-g_\infty\|_{L^q}\to 0$ for $q\gg 1$, we know $$vol_{g_\infty}(\Sigma_\infty)=\lim_{i\to \infty}vol_{g_i}(\Sigma_i)=4\pi$$
        and hence $H_\infty=2$.
        So, we know $\Sigma_i$ is diffeomorphic to the round sphere and
       \begin{align*}
       \lim_{i\to \infty}\int_{\Sigma_i}|H_i|^2d\mu_{g_i}=\lim_{i\to\infty}\int_{\Sigma_i}|H_i-\bar{H}_i|^2d\mu_{g_i}+4\pi\lim_{i\to\infty}\bar{H}^2_i= 16\pi.
       \end{align*}
       Thus by the Gauss-Bonnet formula(or $W^{2,2}$ strong convergence), we have
       \begin{align*}
       \lim_{i\to\infty}\int_{\Sigma_i}|A_i|^2d\mu_{g_i}= 8\pi \text{ and } \lim_{i\to \infty}\int_{\Sigma_i}|A_i^{\circ}|^2d\mu_{g_i}=0,
       \end{align*}
       where $A_i^{\circ}=A_i-\frac{H_i}{2}g_i$ is the trace free part of the second fundamental form.
       Finally, by the results in \cite{dLMu,dLMu2} or \cite[Theorem 1.2.]{LaSc}, we know  there exist conformal parametrizations $\bar{F}_i:\SS^2\to F(\Sigma_i)\subset \RR^3$ with induced metric  $d\bar{F}_i\otimes d\bar{F}_i=e^{2u_i}g_{\SS^{2}}$ such that after appropriate translations and rotations,
       $$\|\bar{F}_i-id_{\SS^2}\|_{W^{2,2}(\SS^2)}+\|u\|_{L^{\infty}(\SS^2)}\le \Psi(\varepsilon).$$

  \end{proof}
  \begin{remark}
  Here we require the codimension to be one for two reasons. Firstly, the unit normal vector field is determined by the tangent map, which is used to get the strong convergence of the mean curvature vectors and then deduce the mean curvature equation of the limit map. Secondly, Alexandrov's classification theorem of CMC hypersurfaces need the codimension one assumption. In fact, there is no good definition of CMC surface in higher codimension. See section \ref{sec: high codimension} for more discussion.
  \end{remark}
\begin{remark}[Half of the expected quantitative estimate] Assume $vol_g(\Sigma)=4\pi$ and $\int_{\Sigma}|H-\bar{H}|^2d\mu_g\le \varepsilon^2$. Then by $\int_{\Sigma}|H|^2d\mu_g\ge 16\pi$ for any closed surfaces $\Sigma$, we know
\begin{align*}
\varepsilon^2\ge \int_{\Sigma}|H-\bar{H}|^2d\mu_g=\int_{\Sigma}|H|^2d\mu_g-4\pi\bar{H}^2\ge 16\pi-4\pi \bar{H}^2
\end{align*}
and hence $\bar{H}\ge 2\sqrt{1-\frac{\varepsilon^2}{16\pi}}\ge 2-\frac{\varepsilon^2}{8\pi}$. So, it is reasonable to ask whether there holds  the upper bound estimate $\bar{H}-2\le C\varepsilon^{2\alpha\le 2},$ from which follows the precise order of the quantitative rigidity:
$$\|\bar{F}_i-id_{\SS^2}\|_{W^{2,2}(\SS^2)}+\|u\|_{L^{\infty}(\SS^2)}\le C\varepsilon^{\alpha}.$$
\end{remark}

In  \cite{dLMu, LaSc}, the global condition on  the smallness of $\int_{\Sigma}|A_\Sigma^0|^2d\mu_g$ is assumed for excluding the collapsing counterexample. The following lemma shows, if the Willmore energy is non-concentrating, then  Simon's monotonicity formula can be used to localize the global Willmore energy restriction
  $$\int_{\Sigma}|H|^2d\mu_g\le 32 \pi(1-\alpha)$$
  to get the local $\gamma$ non-concentration of the density. So, such a restriction provides an alternative global condition to exclude the counterexample as in Figure \ref{figure: collapse}.
   This restriction is necessary, noting $\int_{\Sigma}|H|^2d\mu_g\to 32\pi$ as the neck of  Figure \ref{figure: collapse} goes to zero.
\begin{lemma}\label{non-concentration criteria} Assume $\gamma<\alpha \in (0,\frac{1}{2})$ and  $F_i:\Sigma_i\subset \RR^{n}$ is a sequence of embedding of closed  surfaces  with induced metric $g_i=dF_i\otimes dF_i$,
\begin{align*}
\int_{\Sigma_i}|\vec{H}_i|^2d\mu_{g_i}\le 32\pi(1-\alpha).
\end{align*}
and non-concentrated Willmore energy, i.e.,
$$\lim_{i\to \infty}\lim_{r\to0}\sup_{p_i\in\Sigma_i}\int_{B_r(p_i)}|\vec{H}_i|^2d\mu_{g_i}=0$$
Then there exists $a>0$ such that
$$r^D_\gamma(\Sigma_i)\ge a, \forall i\gg 1.$$
\end{lemma}
\begin{proof}
By the non-concentration condition for the Willmore energy,  there exists $a>0$ such that for sufficient large $i$ and any $p_i\in \Sigma_i$, there holds
\begin{align*}
\int_{B_a(p_i)}|\vec{H}_i|^2d\mu_{g_i}\le \frac{\pi^2(\alpha-\gamma)^2}{4}=:\kappa^2.
\end{align*}
By Simon's monotonicity formula, we know for any $\sigma\le a$, there holds
\begin{align*}
\frac{\mu_{g_i}(B_\sigma(p_i))}{\sigma^2}
&\le \frac{1}{16}\int_{\Sigma_i}|\vec{H}_i|^2d\mu_{g_i}+\frac{1}{2\sigma}\int_{B_\sigma(p_i)}|\vec{H}_i|d_{\mu_{g_i}}\\
&\le \frac{1}{16}\int_{\Sigma_i}|\vec{H}_i|^2d\mu_{g_i}+\frac{1}{2}(\int_{B_\sigma(p_i)}|\vec{H}_i|^2d_{\mu_{g_i}})^{\frac{1}{2}}(\frac{\mu_{g_i}(B_\sigma(p_i))}{\sigma^2})^{\frac{1}{2}}.
\end{align*}
Thus by Young inequality and the Willmore energy bound, we get $\frac{\mu_{g_i}(B_\sigma(p_i))}{\sigma^2}\le42$. Substitute this into the above inequality again. Then we get
\begin{align*}
\frac{\mu_{g_i}(B_\sigma(p_i))}{\sigma^2}\le \frac{1}{16}\int_{\Sigma_i}|\vec{H}_i|^2d\mu_{g_i}+ 4\kappa\le 2\pi(1-\gamma), \forall \sigma\le a,
\end{align*}
which means $r_\gamma^D(\Sigma_i)\ge a$.
\end{proof}

For the sequence of $L^2$-almost CMC surfaces and bounded Willmore energy, the Willmore energies will not concentrate, from which Theorem \ref{main theorem} follows.

  \begin{proof}[ \textbf{Proof of Theorem \ref{main theorem}}]
  We argue by contradiction. If the conclusion was not true, then there is a sequence of immersions $F_i:\Sigma_i\to \RR^3$ of closed surfaces with
  $$\int_{\Sigma_i}|\vec{H}_i|^2d\mu_{g_i}\le 32(1-\alpha)\text{ and } \int_{\Sigma_i}|\vec{H}_i-\bar{H}_i|^2d\mu_{g_i}=\varepsilon_i^2\to 0,$$
   but $\Sigma_i$ not homeomorphic to $S^2$ or there does not exist conformal parametrizations  $\bar{F}_i:\SS^2\to F(\Sigma_i)\subset \RR^3$ with induced metrics  $d\bar{F}_i\otimes d\bar{F}_i=e^{2u_i}g_{\SS^{2}}$ such that after appropriate translations and rotations,
       \begin{align}\label{main estimate}
       \|\bar{F}_i-id_{\SS^2}\|_{W^{2,2}(\SS^2)}+\|u\|_{L^{\infty}(\SS^2)}\le \Psi(\varepsilon).
         \end{align}
         Since all the quantities involved are scaling invariant, we can scale such that $\mu_{g_i}(\Sigma_i)=4\pi$.  By Lemma \ref{non-concentration radius}, for any $\gamma<\alpha\in (0,\frac{1}{2})$, we have
   \begin{align*}
   \int_{B_r(p)}|\vec{H}_i|^2d\mu_{g_i}\le 2\varepsilon_i^2+2^{10}\pi(1-\alpha)^2r^2, \forall p\in \Sigma_i, r\le a(W,\gamma)=\frac{\gamma}{2(1-\alpha)[32\pi(1-\gamma)+7]},
   \end{align*}
   and hence
   $$\lim_{i\to \infty}\lim_{r\to0}\sup_{p_i\in\Sigma_i}\int_{B_r(p_i)}|\vec{H}_i|^2d\mu_{g_i}=0.$$
   So, by Lemma \ref{non-concentration criteria}, we know there exists $a>0$ such that
   \begin{align*}
   r_\gamma^D(\Sigma_i)\ge a, \forall i\gg 1.
   \end{align*}
   Then, by Proposition \ref{convergence}, we know $\Sigma_i$ is homeomorphic to $\SS^2$ for $i$ sufficient large, and there exist conformal parametrizations  $\bar{F}_i:\SS^2\to F(\Sigma_i)\subset \RR^3$ with induced metrics  $d\bar{F}_i\otimes d\bar{F}_i=e^{2u_i}g_{\SS^{2}}$ such that (\ref{main estimate}) holds,  which contradicts to our assumption.
  \end{proof}

As another application, below we study the sequence of domains occurs in \cite[Theorem 1.1]{DeMaMiNe}. The following lemma shows, after a scaling with uniformly upper and lower bound,  such a sequence satisfies all the conditions of Proposition \ref{convergence}, except for the $\gamma$ non-concentration assumption. Here we do not need the diameter bound and mean convex assuption.

\begin{lemma}\label{check condition} For $V>0$, let $\Omega\subset \RR^{3}$ be a bounded domain with  smooth boundary $\Sigma=\partial \Omega$ satisfying
\begin{align}\label{condition}
H^0_\Omega:=\frac{2\mathcal{H}^2(\Sigma)}{3\mathcal{H}^3(\Omega)}=2 \text{ and } \mathcal{H}^2(\Sigma)\le V.
\end{align}
Then, for $\delta_2(\Omega):=\big( \fint_{\Sigma}|\frac{H_{\Sigma}}{H^0_{\Omega}}-1|^2d\mathcal{H}^2\big)^{\frac{1}{2}}$, there exists $c\in [(\frac{4\pi}{V})^{\frac{1}{2}}, 1]$ such that  $\Sigma_c:=c\Sigma$  satisfies
$
\mathcal{H}^2(\Sigma_c)=4\pi ,
$
\begin{align*}
\int_{\Sigma_c}|H_{\Sigma_c}|^2d\mathcal{H}^2\le 4(1+\delta_2(\Omega))^2V
\end{align*}
and
\begin{align*}
 \int_{\Sigma_c}|H_{\Sigma_c}-\bar{H}_{\Sigma_c}|^2d\mathcal{H}^2\le 4V \delta^2_2(\Omega),
\end{align*}
where $H_\Sigma=\langle \vec{H}_\Sigma, \vec{N}_\Sigma\rangle$ is the mean curvature of $\Sigma$ with respect to the  unit inner-pointed  normal vector field $\vec{N}_\Sigma$.
\end{lemma}
\begin{proof}
By the  definition of $\delta_2(\Omega)$ and (\ref{condition}),  we know
\begin{align*}
\delta_2(\Omega)=\big( \fint_{\Sigma}|\frac{H_{\Sigma}}{H^0_{\Omega}}-1|^2d\mathcal{H}^2\big)^{\frac{1}{2}}=\frac{1}{H^0_\Omega}\big(\fint_{\Sigma}|H_\Sigma-H^0_\Omega|^2d\mathcal{H}^2\big)^{\frac{1}{2}}\ge \frac{1}{2\sqrt{V}}\big(\int_{\Sigma}|H_\Sigma-\bar{H}_{\Sigma}|^2d\mathcal{H}^2\big)^{\frac{1}{2}}.
\end{align*}
By (\ref{condition}), the definition of $H^0_\Omega$ and the isoperimetric inequality, we know
\begin{align*}
2=H^0_\Omega=\frac{2\mathcal{H}^2(\Sigma)}{3\mathcal{H}^3(\Omega)}\ge \frac{2\cdot 6\sqrt{\pi}\mathcal{H}^2(\Sigma)}{3(\mathcal{H}^2(\Sigma))^{\frac{3}{2}}}=\frac{4\sqrt{\pi}}{(\mathcal{H}^2(\Sigma))^{\frac{1}{2}}}.
\end{align*}
This combines (\ref{condition}) imply $\mathcal{H}^2(\Sigma)\in  [4\pi,V]$. Take $c=\big(\frac{4\pi}{\mathcal{H}^2(\Sigma)}\big )^{\frac{1}{2}}\in [(\frac{4\pi}{V})^{\frac{1}{2}}, 1]$ and define $\Sigma_c=c\Sigma$. Then we get

\begin{align*}
\mathcal{H}^2(\Sigma_c)=c^2\mathcal{H}^2(\Sigma)=4\pi,
\end{align*}
and
 \begin{align*}
  \int_{\Sigma_c}|H_{\Sigma_c}-\bar{H}_{\Sigma_c}|^2d\mathcal{H}^2=\int_{\Sigma}|H_\Sigma-\bar{H}_{\Sigma}|^2d\mathcal{H}^2\le 4V \delta^2_2(\Omega).
  \end{align*}
Moreover, by (\ref{condition}) again, we know
\begin{align*}
\int_{\Sigma_c}|H_{\Sigma_c}|^2d\mathcal{H}^2&=\int_{\Sigma}|H_\Sigma|^2d\mathcal{H}^2\\
&\le [(\int_{\Sigma}|H_\Sigma-H^0_\Omega|^2d\mathcal{H}^2)^{\frac{1}{2}}+H^0_\Omega(\mathcal{H}^2(\Sigma))^{\frac{1}{2}}]^2\\
&\le 4(1+\delta_2(\Omega))^2V.
\end{align*}
\end{proof}
So, if we add the $\gamma$ non-concentration assumption, we get the following compactness result for the sequence of surfaces considered in \cite[Theorem 1.1]{DeMaMiNe}.
\begin{corollary}\label{DMMN setting}
  For $V>0$, let $\{\Omega_h\}_{h\in \NN}$ be a sequence of  bounded domain in $\RR^{3}$ with  smooth boundary $\Sigma_h=\partial \Omega_h$ and  satisfying
\begin{align*}
H^0_{\Omega_h}=2 , \  \mathcal{H}^2(\Sigma_h)\le V \text{ and } \lim_{h\to \infty} \delta_2(\Omega_h)=0.
\end{align*}

Assume $V<8\pi$,  or  more general,  there exists $\gamma\in (0,\frac{1}{2})$ and $a>0$ such that
 \begin{align}\label{assumption}
  r^D_{\gamma}(\Sigma_h)\ge a.
 \end{align}
  Then, there exist conformal parametrizations $f_h: {\SS}^2\to \Sigma_h\subset \RR^3$  such that after appropriate translations and rotations,
\begin{align}\label{strong convergence}
\lim_{h\to \infty}\|f_h-id_{{\SS}^2}\|_{W^{2,2}}+\lim_{h\to \infty}\|u_h\|_{L^{\infty}({\SS}^2)}=0,
\end{align}
where $id_{{\SS}^2}$ is the including map of the round sphere ${\SS}^2\subset \RR^3$ and  $e^{2u_{h}}g_{{\SS}^2}=f_h^{*}g_{\RR^3}$.
\end{corollary}
\begin{proof}
Define $c_h=\big(\frac{4\pi}{\mathcal{H}^2(\Sigma_h)}\big )^{\frac{1}{2}}$ and $\Sigma_{h,c_h}=c_h\Sigma_h$.  By lemma \ref{check condition}, we know $$\mathcal{H}^2(\Sigma_{h,c_h})=4\pi\text{ ,  }  \int_{\Sigma_{h,c_h}}|H_{\Sigma_{h,c_h}}|^2d\mathcal{H}^2\le 4(1+\delta_2(\Omega_h))^2V$$
and
$$\lim_{h\to\infty}\int_{\Sigma_{h,c_h}} |H_{\Sigma_{h,c_h}}-\bar{H}_{\Sigma_{h,c_h}}|^2d\mathcal{H}^2=0.$$
If $V<8\pi$, then we know
$$\lim_{h\to\infty} \int_{\Sigma_{h,c_h}}|H_{\Sigma_{h,c_h}}|^2d\mathcal{H}^2\le 4V<32 \pi.$$
So, (\ref{assumption}) follows by the same argument as in the proof of Corollary \ref{topology rigidity}.  This means (\ref{assumption}) is a more general assumption than $V<8\pi$.
Moreover, since $c_h\ge (\frac{4\pi}{V})^{\frac{1}{2}}$, we know
$$r_\gamma^D(\Sigma_{h,c_h})=c_hr_{\gamma}^D(\Sigma_h)\ge (\frac{4\pi}{V})^{\frac{1}{2}}a>0.$$
So, by Proposition \ref{convergence}, we know (\ref{strong convergence}) holds for $\Sigma_{h,c_h}$ and $\lim_{h\to \infty}\bar{H}_{\Sigma_{h,c_h}}=2$. To show (\ref{strong convergence}) holds for $\Sigma_{h}$, we only need to check $\lim_{h\to \infty}c_h=1$.

Noting $\bar{H}_{\Sigma_{h,c_h}}=\frac{1}{c_h}\bar{H}_{\Sigma_h}$ and
\begin{align*}
|\bar{H}_{\Sigma_h}-H^0_{\Omega_h}|
&\le\big(\fint_{\Sigma_h}|H_{\Sigma_h}-H^0_{\Omega_h}|^2d\mathcal{H}^2\big)^{\frac{1}{2}}+\big(\fint_{\Sigma_h}|H_{\Sigma_h}-\bar{H}_{\Sigma_h}|^2d\mathcal{H}^2\big)^{\frac{1}{2}}\\
&\le2\big(\fint_{\Sigma_h}|H_{\Sigma_h}-H^0_{\Omega_h}|^2d\mathcal{H}^2\big)^{\frac{1}{2}}\\
&\to 0,
\end{align*}
we get
$$\lim_{h\to \infty}c_h=\lim_{h\to \infty}\frac{\bar{H}_{h,c_h}}{\bar{H}_{\Sigma_h}}=\frac{2}{H^0_{\Omega_h}}=1.$$

\end{proof}

\section{higher codimension}\label{sec: high codimension}
There is no natural conception of CMC surfaces in higher codimension since the direction of the mean curvature vector is no more determined by the tangent bundle.  Motivated by the estimate(after translation and rotation)\cite[Proposition 2.4]{LaSc}
\begin{align*}
\|\vec{H}_{\Sigma}+2id_{\Sigma}\|_{L^2(\Sigma)}\le C_n\|A^0_\Sigma\|_{L^2(\Sigma)}
\end{align*}
for surface $\Sigma\subset \RR^n$ with $\mathcal{H}^2(\Sigma)=4\pi$ and $\|A^0_\Sigma\|_{L^2(\Sigma)}\le \delta_n$, we consider the quantity
$$\mathcal{J}(F)=\min_{c\in \RR, \beta\in \RR^{n}}\int_{\Sigma}|\vec{H}+cF+\beta|^2d\mu_{g}$$
for immersion $F:\Sigma\to \RR^n$.  This means, instead of comparing the mean curvature vector with constant times of some unit normal vector field(may not exist), we compare the mean curvature vector with the position vector, which is always well-defined.  Direct calculation shows the minimum is attained by
$$c=-\frac{\int_{\Sigma}\langle \vec{H},F\rangle d\mu_g}{\int_{\Sigma}|F|^2d\mu_g}\  \text{  and  } \ \beta=-\int_{\Sigma}\vec{H}d\mu_g=-\int_{\Sigma}\Delta_{g}Fd\mu_g=0$$
and
$$\mathcal{J}(F)=\int_{\Sigma}|\vec{H}-\bar{H}^F|^2d\mu_{g},$$
where $\bar{H}^F=\frac{\int_{\Sigma}\langle \vec{H},F\rangle d\mu_g}{\int_{\Sigma}|F|^2d\mu_g}F$. Before considering the quantitative rigidity, we first need to classify surfaces in $\RR^n$ satisfying $\mathcal{J}(F)=0$, which holds if and only if
$\vec{H}=cF.$
Such surfaces can be regarded as  "constant mean curvature surface in high codimension". But it is well known they are just minimal surfaces in the sphere. More precisely,
 \begin{lemma}\label{minimal surface equation}
Assume $F:\Sigma\to \RR^n$ is a smooth immersion (here $\Sigma$ is not assumed to be complete) and $g=dF\otimes dF$. Then
$$\vec{H}=cF $$
if and only if
\begin{itemize}
\item
$c=0$ and $F(\Sigma)$  is a piece of a minimal surface in ${\RR}^n$,
\item  or  $c<0$ and  $F(\Sigma)$ is a piece of a minimal surface in the sphere ${\SS}^{n-1}(r)$ for $r=\sqrt{-\frac{2}{c}}$.
\end{itemize}
 \end{lemma}
 \begin{proof}
 We only need to prove the case $c\neq 0$.
 On the one hand, if $F(\Sigma)$ is a minimal surface in ${\SS}^{n-1}(r)$, then we know $|F|=r$ and $\vec{H}^{{\SS}^{n-1}(r)}=0$. Thus
 \begin{align*}
 \vec{H}&=\vec{H}^{{\SS}^{n-1}(r)}+\langle \vec{H},\frac{F}{r}\rangle \frac{F}{r}\\
 &=\frac{1}{r^2}\langle \Delta_gF,F\rangle F\\
 &=\frac{1}{r^2}(div\langle \nabla_gF,R\rangle-|\nabla_gF|^2)F\\
 &=\frac{1}{r^2}(\frac{1}{2}\Delta_g|F|^2-g^{ij}\langle F_i,F_j\rangle)F\\
 &=-\frac{2}{r^2}F.
 \end{align*}
 On the other hand, if $\vec{H}=cF$ for some $c\neq 0$, then, by $\langle \frac{\partial F}{\partial x^i}, \vec{H}\rangle\equiv 0$, we know
 \begin{align*}
 d|F|^2=2\langle dF,F\rangle=\frac{2}{c}\langle dF,\vec{H}\rangle=0.
 \end{align*}
 This means $|F|^2$ is a constant, i.e., $F(\Sigma)$ is a surface in ${\SS}^{n-1}(r)$ for some $r>0$. Thus by
 $$\vec{H}=\vec{H}^{{\SS}^{n-1}(r)}\oplus\langle \vec{H},\frac{F}{|F|}\rangle \frac{F}{|F|},$$
 we know $\vec{H}^{{\SS}^{n-1}(r)}=0$, i.e., $F(\Sigma)$ is a piece of minimal surface in ${\SS}^{n-1}(r)$. Moreover, by the first part of the proof, we know
 \begin{align*}
 cF=\vec{H}=-\frac{2}{r^2}F.
 \end{align*}
 This implies $c<0$ and $r=\sqrt{-\frac{2}{c}}$.
 \end{proof}
For the expected quantitative rigidity, the main difficulty is we do not know whether  the Willmore energy will concentrate.  The following lemma shows this will not happen.
\begin{lemma}[Non-concentration of the Willmore energy]\label{lemma: non-concentration of Willmore} Let $F:\Sigma\to \RR^n$ be a smooth immersion of a closed surface $\Sigma$ with the induced metric $g=dF\otimes dF$. Assume
\begin{align*}
\frac{1}{4}\int_{\Sigma}|\vec{H}|^2d\mu_g\le W
\end{align*}
and there exists $c\in \RR$ such that
\begin{align*}
\int_{\Sigma}|\vec{H}+cF|^2d\mu_{g}=\min_{\alpha \in \RR }\int_{\Sigma}|\vec{H}+\alpha F|^2d\mu_{g}\le \varepsilon^2<1.
\end{align*}
Then, for any $p\in \Sigma$ and $r>0$, there holds
\begin{align*}
\int_{B_r(p)}|\vec{H}|^2 d\mu_g \le 2\varepsilon^2+2^{21}\frac{W^4}{\mu_g(\Sigma)}r^2.
\end{align*}
Moreover, for $\gamma\in (0,\frac{1}{2})$, if $\varepsilon\le \frac{\gamma}{6}$ and there exists $a\le a(\gamma, \mu_g(\Sigma),W):=\frac{\sqrt{\mu_g(\Sigma)}\gamma}{6720W^2}$ such that
\begin{align*}
\frac{\mu_g(B_a(p))}{\pi a^2}\le 2(1-\gamma), \forall p\in \Sigma,
\end{align*}
then
\begin{align*}
r_{\frac{\gamma}{2}}^D\ge a.
\end{align*}
\end{lemma}
\begin{proof}
Step 1. The upper and lower bound of $c$.

Since
\begin{align*}
\int_{\Sigma}|\vec{H}+cF|^2d\mu_{g}=\min_{\alpha\in \RR }\int_{\Sigma}|\vec{H}+\alpha F|^2d\mu_{g}\le \varepsilon^2,
\end{align*}
we know $c=-\frac{\int_{\Sigma}\langle \vec{H},F\rangle d\mu_g}{\int_{\Sigma}|F|^2d\mu_g}=\frac{\int_{\Sigma}|\nabla_gF|^2d\mu_g}{\int_{\Sigma}|F|^2d\mu_g}=\frac{2\mu_g(\Sigma)}{\int_{\Sigma}|F|^2d\mu_g}>0$
and hence
\begin{align*}
\int_{\Sigma}|\vec{H}|^2d\mu_g-2c\mu_g(\Sigma)=\int_{\Sigma}|\vec{H}|^2d\mu_g-c^2\int_{\Sigma}|F|^2d\mu_g=\int_{\Sigma}|\vec{H}+cF|^2d\mu_g\in[0,\varepsilon^2],
\end{align*}
which implies
\begin{align}\label{bound of constant}
\frac{16\pi-\varepsilon^2}{2\mu_g(\Sigma)}\le\frac{\int_{\Sigma}|\vec{H}|^2d\mu_g-\varepsilon^2}{2\mu_g(\Sigma)} \le  c\le \frac{\int_{\Sigma}|\vec{H}|^2d\mu_g}{2\mu_g(\Sigma)}\le\frac{2W}{\mu_g(\Sigma)}.
\end{align}
Step 2. The upper bound of $|F|$.

By \cite[Lemma 1.1]{S}, we know
\begin{align*}
\max_{x,y\in \Sigma}|F(x)-F(y)|\le diam(F(\Sigma))\le 28\sqrt{\mu_g(\Sigma)W}.
\end{align*}
Noting $\int_{\Sigma}\vec{H}d\mu_g=0$ and using the lower bound estimate of $c$  in Step 1., we get
\begin{align*}
|\fint_{\Sigma}Fd\mu_g|=\frac{1}{c\mu_g(\Sigma)}|\int_{\Sigma}(\vec{H}+cF)d\mu_g|\le \frac{2\varepsilon (\mu_g(\Sigma))^{\frac{1}{2}}}{16\pi-\varepsilon^2}.
\end{align*}
Combing the above two inequalities together, we get
\begin{align*}
|F(x)|&\le |\fint_{\Sigma}(F(x)-F(y))d\mu_g(y)|+|\fint_{\Sigma}F(y)d\mu_g(y)|\\
&\le (\mu_g(\Sigma))^{\frac{1}{2}}(28W^{\frac{1}{2}}+\frac{2\varepsilon}{16\pi-\varepsilon^2})\\
&\le 30(\mu_g(\Sigma)W)^{\frac{1}{2}}.
\end{align*}
Step 3. Non-concentration of the Willmore energy.

By Simon's monotonicity formula (\ref{mono-ineq}), we know  for any $p\in \RR^n$ and $r>0$, there holds
\begin{align*}
\mu_g(B_r(p))\le \frac{64W}{3}r^2 .
\end{align*}
Thus by the upper bound of $c$ in Step 1. and the upper bound of $|F|$ in step 2., we know for any $p\in \Sigma$ and $r>0$, there holds
\begin{align*}
\int_{B_r(p)}|\vec{H}|^2d\mu_g&\le 2\int_{\Sigma}|\vec{H}+cF|^2d\mu_g+2c^2\int_{B_r(p)}|F|^2d\mu_g\\
&\le 2\varepsilon^2 +450\cdot 2^8 \frac{W^3}{\mu_g(\Sigma)}\mu_g(B_r(p))\\
&\le 2\varepsilon^2+2^{21}\frac{W^4}{\mu_g(\Sigma)}r^2.\\
\end{align*}
So, when $\varepsilon\le \frac{\gamma}{6}$ and $r\le a(\gamma, \mu_g(\Sigma),W)=\frac{\sqrt{\mu_g(\Sigma)}\gamma}{6720W^2}$, we have
$$\int_{B_r(p)}|\vec{H}|^2d\mu_g\le \frac{\gamma^2}{9}.$$
Taking $\delta=\frac{\gamma}{3}$ in the monotonicity inequality (\ref{mono-ineq}), we get
\begin{align*}
\frac{\mu_g(B_r(p))}{\pi r^2}\le 2(1+\delta)(1-\gamma)+(\frac{3}{16\pi}+\frac{1}{4\pi\delta})\delta^2\le 2(1-\frac{\gamma}{2}), \forall p\in \Sigma, r\le a.
\end{align*}
\end{proof}
For convenience of dealing with compactness in higher codimension, we give a notion for the space of closed embedding surfaces with normalized area, uniformly bounded Willmore energy and non-concentratign dendity. More precisely, for  $\gamma\in (0,\frac{1}{2})$, $a,V,W>0$ and integer $n\ge 3$, define
\begin{align*}
  \mathcal{E}_n&(W,V,\gamma,a)=\bigg\{F:\Sigma\to {\RR}^{n}| \Sigma \text{ is a closed surface}, F \text{ is an embedding},g=dF\otimes dF,\\
   &\mu_g(\Sigma)=V,\ \frac{1}{4}\int_{\Sigma}|\vec{H}|^2d\mu_g\le W,\ \frac{\mu_g(B_a(p))}{\pi a^2}\le 2(1-\gamma), \forall p\in \Sigma.\bigg\}
\end{align*}

\begin{remark}The following is a version of Proposition \ref{convergence} in higher codimension case, which means any sequence in $\mathcal{E}_n(W,V,\gamma, a)$ with $\mathcal{J}(F_i)\to 0$ will converge to a  smooth minimal surface in the sphere.  The requirement $\mu_{g_i}(\Sigma_i)=V$ is no restricted since we can always scale $\Sigma_i$ such that $\mu_{g_i}(\Sigma_i)=V$ holds.
\end{remark}

\begin{proposition}\label{convergence 2} Let $n\ge 3$, $\gamma\in (0,\frac{1}{2})$,  $V, W>0$ and $0<a<\frac{\sqrt{V}\gamma}{6720W^2})$. Assume there is a sequence of $F_i:\Sigma_i\to \RR^n$ belonging to $\mathcal{E}_n(W,V,\gamma, a)$ such that
\begin{align*}
\mathcal{J}(F_i)=\int_{\Sigma_i}|\vec{H}_i+c_iF_i|^2d\mu_{g_i}\le \varepsilon_i^2\to 0.
\end{align*}
Then  $c_i\to c_\infty\in [\frac{8\pi}{V},\frac{2W}{V}]$ and  there exists an embedded minimal surface $F_\infty:\Sigma_\infty\to {\SS}^{n-1}(R_\infty)$ for $R_\infty=\sqrt{\frac{2}{c_\infty}})$ such that $$\mu_{g_\infty}(\Sigma_\infty)=V,$$ $$\lim_{i\to\infty}\int_{\Sigma_i}|\vec{H}_i|^2d\mu_{g_i}
=\int_{\Sigma_\infty}|\vec{H}_\infty|^2d\mu_{g_\infty}=2c_\infty V$$
and $F_i$ converges to $F$ in $W^{2,2}$-topology. More precisely, there exists $C>0$ and diffeomorphisms $\varphi_i: \Sigma_\infty\to \Sigma_i$ such that
\begin{enumerate}[1)]
\item $\|\varphi_i^*g_i-g_\infty\|_{L^q(\Sigma_\infty)}\to 0, \text{ for }  q\gg 1$ and
\begin{align*}
C^{-1}g_\infty\le \varphi_i^*g_i\le C g_\infty;
\end{align*}
\item the complex structure of $\Sigma_i$ converges to the complex structure of $\Sigma_\infty;$
\item $\|F_i\circ \varphi_i-F_\infty\|_{W^{2,2}(\Sigma_\infty)}\to 0$.
\end{enumerate}
\end{proposition}

\begin{proof}
Step 1.  By Lemma \ref{lemma: non-concentration of Willmore}, we know for $i$ sufficient large,
\begin{align*}
r_{\frac{\gamma}{2}}^D\ge a \text{ and } \int_{B_r(p)}|\vec{H}_i|^2d\mu_{g_i}\le 2\varepsilon_i^2+2^{21}\frac{W^4}{V}\rho^2, \forall \rho\le a.
\end{align*}
By the same argument as the proof of Proposition \ref{convergence}, we know there exists $r>0$ such that for any $p_i\in \Sigma_i$, there exists an isothermal coordinate $\varphi_i:D_r(0)\to U_i\supset B_r^{\Sigma_i}(p_i)$ with $\varphi_i(0)=p_i$, $\varphi_i^*g_i=e^{2u_i}(dx^1\otimes dx^1+dx_2\otimes dx^2)$ and $\tilde{F}_i=F_i\circ \varphi_i-F_i(p_i)$ satisfying
\begin{align}\label{uniform bound}
\|u_i\|_{C^0(D_r)}+\|Du_i\|_{L^2(D_r)}\le \Psi(\varepsilon) \text{  }
\|\tilde{F}_i\|_{W^{2,2}(D_{\frac{3}{4}r})\le C(r)}.
\end{align}
and
\begin{align}\label{equation}
\Delta \tilde{F}_i=\vec{H}_ie^{2u_i}.
\end{align}
Moreover, there exist $u_\infty\in C^0(D_{\frac{3}{4}r})$ and $\tilde{F}_\infty\in W^{2,2}(D_{\frac{3}{4}r})$ such that
\begin{align}\label{uniform bound 2}
\|u_\infty\|_{C^0(D_{\frac{3}{4}r})}+\|Du\|_{L^2(\frac{3}{4}r)}\le \Psi(\varepsilon)  ,\ \  d\tilde{F}_\infty\otimes d\tilde{F}_{\infty}=g_\infty=e^{2u_\infty}(dx^1\otimes dx^1+dx^2\otimes dx^2)
\end{align}
and
\begin{align}\label{weak limit}
\|u_i-u_\infty\|_{L^q(D_{\frac{3}{4}r})}+\|\tilde{F}_i-\tilde{F}_\infty\|_{W^{1,q}(D_{\frac{3}{4}r})}\to 0,  \forall q\gg 1.
\end{align}
By (\ref{bound of constant}), we know $c_i$ are bounded and hence converges to a limit $c_\infty \in [\frac{8\pi}{V},\frac{W}{V}]$ after passing to a subsequence. Thus by the dominated convergence theorem and (\ref{uniform bound})(\ref{uniform bound 2})(\ref{weak limit}), we know
\begin{align*}
\lim_{i\to \infty}\|\vec{H}_ie^{2u_i}+c_\infty\tilde{F}_\infty e^{2u_\infty}\|_{L^2(D_{\frac{3}{4}r})}=0.
\end{align*}
By (\ref{equation}) and (\ref{weak limit}), we know on $D_{\frac{3}{4}r}$, $\tilde{F}_\infty$ satisfies the weak equation
\begin{align}\label{limit equation}
\Delta \tilde{F}_\infty=-c_\infty \tilde{F}_\infty e^{2u_\infty}.
\end{align}
By elliptic regularity, we have the strong convergence and
\begin{align*}
\lim_{i\to \infty}\|\tilde{F}_i-\tilde{F}_\infty\|_{W^{2,2}(D_{\frac{r}{2}})}=0
\end{align*}
and
 $\tilde{F}_\infty\in W^{2,q}, \forall q\gg1$ and hence (by  (\ref{limit equation}) and (\ref{uniform bound 2})) ,
\begin{align*}
\Delta d\tilde{F}_\infty=-c_\infty d(\tilde{F}_\infty e^{2u_\infty})\in L^2(D_{\frac{3}{4}r}).
\end{align*}
Thus $d\tilde{F}_\infty\in W^{2,2}_{loc}(D_{\frac{3}{4}r})$ and hence by the conformal condition $d\tilde{F}_\infty\otimes d\tilde{F}_\infty=e^{2u_\infty}(dx^1\otimes dx^1+dx^2\otimes dx^2)$, we know $e^{2u_\infty}\in C^{\beta}$ for any $\beta\in (0,\frac{1}{2})$. Now, combining the limit equation (\ref{limit equation}) and the conformal condition, by elliptic regularity and bootstrap argument, we know $F_\infty, e^{2u_\infty}\in C^\infty(D_{\frac{3}{4}r})$.  So, (\ref{limit equation}) means
\begin{align*}
\vec{H}_\infty=\Delta_{g_\infty}\tilde{F}_\infty=-c_\infty \tilde{F}_\infty
\end{align*}
and hence by Lemma \ref{minimal surface equation} , we know $\tilde{F}_\infty(D_{\frac{3}{4}r})$ is a piece of minimal surface in ${\SS}^{n-1}(R_\infty)$, where the radius $R_\infty=\sqrt{\frac{2}{c_\infty}}$.
Again by \cite{Sun-Zhou}, argue as in the proof of Proposition \ref{convergence}, we know there is an embedding $F_\infty:\Sigma_\infty\to {\SS}^{n-1}(R_\infty)$ of a closed surface $\Sigma_\infty$ such that $F_\infty(\Sigma_\infty)$ is a minimal surface in ${\SS}^{n-1}(R_\infty)$ and item $1)2)3)$ of the conclusion holds.
So, $||F_i\circ \varphi_i|-R_\infty|_{C^0(\Sigma_\infty)}\to 1$ and hence
\begin{align*}
0=\lim_{i\to \infty} \mathcal{J}(F_i)
&=\lim_{i\to\infty}\int_{\Sigma_i}|\vec{H}_i|^2d\mu_{g_i}-{c_i^2}\int_{\Sigma_i}|F_i|^2d\mu_{g_i}\\
&=\lim_{i\to\infty}\int_{\Sigma_i}|\vec{H}_i|^2d\mu_{g_i}-c_\infty^2\int_{\Sigma_\infty}|F_\infty|^2d\mu_{g_\infty},
\end{align*}
which implies
\begin{align*}
\lim_{i\to\infty}\int_{\Sigma_i}|\vec{H}_i|^2d\mu_{g_i}=\int_{\Sigma_\infty}|\vec{H}_\infty|^2d\mu_{g_\infty}=2c_\infty V.
\end{align*}
Finally,  $\mu_{g_\infty}(\Sigma_\infty)=\lim_{i\to \infty}\mu_{g_i}(\Sigma_i)=V$ follows directly from the $L^q$ convergence of the metric for $q>2$.
 \end{proof}
 When restricting the topology of $\Sigma$ to be a sphere and the Willmore energy not exceeding $8\pi$, we get the bi-Lipsschitz quantitative rigidity of the round sphere as minimal surface in ${\SS}^{n-1}$.
 \begin{corollary}\label{quantitative minimal}For any $W<8\pi$, there exists $\varepsilon=\varepsilon(W)>0$ such that if an embedding  $F: \Sigma \to \RR^n$ of a sphere type surface $\Sigma\cong S^2$ with induced metric $g=dF\otimes dF$ satisfies
\begin{align*}
\mu_g(\Sigma)=4\pi, \frac{1}{4}\int_{\Sigma}|\vec{H}|^2d\mu_g\le W<8\pi ,
\end{align*}
and there exists a constant $c\in \RR$ such that
\begin{align*}
\mathcal{J}(F)=\int_{\Sigma}|\vec{H}+cF|^2d\mu_g\le \varepsilon^2,
\end{align*}
then there exists conformal diffeomorphism  $\tilde{F}: S^2\to F(\Sigma)\subset \RR^n$ such that   $d\tilde{F}\otimes d\tilde{F}=e^{2u}g_{{\SS}^2}$ and
\begin{align}\label{conclusion}
|c-2|+\|u-1\|_{C^0({\SS}^2)}+\|\tilde{F}-id_{{\SS}^2}\|\le \Psi(\varepsilon).
\end{align}
\end{corollary}
\begin{proof}
We argue by contradiction. Assume the conclusion doest not hold. Then there are a sequence of embeddings $F_i:\Sigma_i\to \RR^n$ of sphere $\Sigma_i\cong S^2$ with induced metric $g_i=dF_i\otimes dF_i$ such that
$$\mu_{g_i}(\Sigma_i)=4\pi, \frac{1}{4}\int_{\Sigma_i}|\vec{H}_i|^2d\mu_{g_i}\le W<8\pi$$
and
$$\mathcal{J}(F_i)=\int_{\Sigma_i}|\vec{H}_i+c_iF_i|^2d\mu_{g_i}=\varepsilon_i^2\to 0,$$
but (\ref{conclusion}) does not hold for each $F_i$. By Lemma \ref{non-concentration for mean curvature}, Lemma \ref{non-concentration criteria} and Proposition \ref{convergence 2}, we know $c_i\to c_\infty\in [2,\frac{W}{2\pi}]\subset [2,4)$ and an minimal embedding $F_\infty:\Sigma_\infty\to {\SS}^{n-1}(R_\infty)$ such that $R_\infty=\sqrt{\frac{2}{c_\infty}}\in (\frac{1}{\sqrt{2}},1]$,
$\mu_{g_\infty}(\Sigma_\infty)=4\pi$
and
 \begin{align*}
\lim_{i\to\infty}\int_{\Sigma_i}|\vec{H}_i|^2d\mu_{g_i}=\int_{\Sigma_\infty}|\vec{H}_\infty|^2d\mu_{g_\infty}=8\pi c_\infty.
\end{align*}
 Thus $\tilde{\Sigma}_\infty=R_\infty^{-1}F_\infty(\Sigma_\infty)$ is a minimal surface in ${\SS}^{n-1}(1)$ with area $$\mu_{\tilde{g}_\infty}(\tilde{\Sigma}_\infty)=4\pi R_\infty^{-2}\in [4\pi, 8\pi).$$
 But by the theorems of Calabi\cite{Ca} and Barbosa\cite[Corollary 4.14]{B}, the area of minimal sphere in ${\SS}^{n-1}(1)$ is a multiple of $4\pi$. So, we know
 \begin{align*}
 \frac{1}{4}\int_{\tilde{\Sigma}_\infty}|\vec{\tilde{H}}_\infty|^2d\mu_{\tilde{g}_\infty}
 =\mu_{\tilde{g}_\infty}(\tilde{\Sigma}_\infty)=4\pi,
 \end{align*}
 and hence $R_\infty=1$, $c_\infty=2$ and $\tilde{\Sigma}_\infty=\Sigma_\infty$ is a round sphere $\SS^2$ in ${\SS}^{n-1}$. Finally, since
 $$\lim_{i\to\infty}\int_{\Sigma_i}|\vec{H}_i|^2d\mu_{g_i}=8\pi c_\infty=16\pi,$$
 by Gauss-Bonnet formula, we know
 $
 \int_{\Sigma_i}|A_i^0|^2d\mu_{g_i}\to 0
 $
 and  hence by \cite[Theorem 1.2]{LaSc}, (\ref{conclusion}) holds for $F_i$ when $i$ is sufficient large. This is a contradiction.
\end{proof}
\begin{remark} Since $\Sigma\cong S^2$ implies $$\frac{1}{4}\int_{\Sigma}|\vec{H}|^2d\mu_g\le W\text{ if and only if } \|A^0\|_{L^2(\Sigma)}^2<2(W-4\pi),$$
by \cite[Proposition 3,1]{LaSc},  the existence of the above conformal parametrization follows directly if $W<4\pi+e_n$ for $e_n$ from \cite{Sch13} defined by
\begin{align*}
e_n:=\left\{
\begin{aligned}
4\pi &\ \text{ for } n=3, \\
\frac{8\pi}{3} &\ \text{ for } n=4, \\
2\pi &\ \text{ for } n\ge 5.
\end{aligned}
\right.
\end{align*}
So, the non-trivial part for the above corollary is when $W\in [4\pi+e_n,8\pi)$.
\end{remark}

\begin{remark}  Since there are many minimal surfaces(e.g. Lawson's surface $\xi_{1,g}$) in ${\SS}^{n-1}$ with Willmroe energy smaller than $8\pi$ and nontrivial topology,  the topological assumption $\Sigma\approx S^2$ is necessary in the above result.  When considering surfaces with varying topology, we get the topological finiteness of minimal surfaces in the sphere with area uniformly smaller than $8\pi$.
\end{remark}
\begin{corollary}
For $\alpha>0$, define
\begin{align*}
\mathcal{M}(\alpha, n):=\{\Sigma| \Sigma \text{ is  minimal surface in } {\SS}^{n-1} \text{ with } \mu_g(\Sigma)\le (1-\alpha)8\pi\}.
\end{align*}
Then, any  sequence $\{\Sigma_i\} \subset\mathcal{M}(\alpha,n)$ converges smoothly to a limit $\Sigma_\infty\in \mathcal{M}(\alpha,n)$.
\end{corollary}
\begin{proof}
Assume $\Sigma_i\in \mathcal{M}(n,\alpha)$ is a sequence of minimal surface in ${\SS}^{n-1}$  with area $\mu_{g_i}(\Sigma_i)=V_i\le (1-\alpha)8\pi$.  Then, by $\vec{H}=-2F$, we know that
\begin{align*}
4\pi\le \frac{1}{4}\int_{\Sigma_i}|\vec{H}_i|^2d\mu_{g_i}=\mu_{g_i}(\Sigma_i)\le (1-\alpha)8\pi.
\end{align*}
Letting $d_i=\sqrt{\frac{4\pi}{V_i}}$ and $\tilde{\Sigma}_i=d_i\Sigma_i$, then  $d_i\in (\sqrt{\frac{1}{2(1-\alpha)}},1]$, $\mu_{\tilde{g}_i}(\tilde{\Sigma}_i)=4\pi$ and
$$
\vec{\tilde{H}}=-\frac{2}{d_i^2}\tilde{F}_i.
$$
  Letting $c_i=\frac{2}{d_i^2}$ and noting the Willmore energy is scaling invariant, by the same argument as in the proof of Corollary \ref{quantitative minimal}, we know $d_i\to d_\infty, c_i\to c_\infty=\frac{2}{d_\infty^2}$  and $\tilde{\Sigma}_i$ converges smoothly to a limit minimal surface in $\tilde{\Sigma} _\infty$ in ${\SS}^{n-1}(R_\infty)$ with $R_\infty=\sqrt{\frac{2}{c_\infty}}=d_\infty$ and $\mu_{\tilde{g}_\infty}(\tilde{\Sigma}_\infty)=4\pi$. Letting $\Sigma=\frac{1}{d_\infty}\tilde{\Sigma}$, we know
  \begin{align*}
  \Sigma_i=\frac{1}{d_i}\tilde{\Sigma}_i\to \frac{1}{d_\infty}\tilde{\Sigma}_\infty=\Sigma_\infty,
  \end{align*}
  and $\tilde{\Sigma}_\infty\in \mathcal{M}(\alpha,n)$.
\end{proof}
This corollary is a minimal surface version of the  compactness of Willmore surfaces with Willmore energy uniformly smaller than $8\pi$ \cite{KLS}. The  minimality assumption guarantees us not needing to composite the M\"{o}bius transformations. The sequence of Lawson's minimal surfaces also shows the assumption $\mu_{g_i}(\Sigma_i)\le (1-\alpha)8\pi$ is sharp.

\end{document}